\documentclass{amsart}
\usepackage{amsmath}
\usepackage{amssymb}
\usepackage{enumerate}
\usepackage{enumitem}
\usepackage{mathrsfs}
\usepackage{cases}
\usepackage{color}
\usepackage{tikz}
\usepackage{appendix}
\usepackage{makecell}
\usepackage{multirow}
\usepackage{algpseudocode}
\usepackage{algorithmicx,algorithm}
\usepackage{graphicx}
\definecolor{darkblue}{rgb}{0.0, 0.0, 0.55}
\definecolor{bordeaux}{rgb}{0.34, 0.01, 0.1}
\usepackage[colorlinks,linkcolor=bordeaux,citecolor=darkblue,urlcolor=black,hypertexnames=true]{hyperref}
\usetikzlibrary{arrows,positioning}

\newtheorem{theorem}{Theorem}[section]
\newtheorem{lemma}[theorem]{Lemma}
\newtheorem{definition}[theorem]{Definition}
\newtheorem{example}[theorem]{Example}

\newtheorem{proposition}[theorem]{Proposition}

\newtheorem{remark}[theorem]{Remark}
\numberwithin{equation}{section}
\def\Z{{\mathbb{Z}}}

\def\R{{\mathbb{R}}}
\def\CC{{\mathbb{C}}}
\def\N{{\mathbb{N}}}

\def\x{{\mathbf{x}}}

\def\y{{\mathbf{y}}}

\def\a{{\boldsymbol{\alpha}}}

\def\b{{\boldsymbol{\beta}}}
\def\g{{\boldsymbol{\gamma}}}
\def\up{{\boldsymbol{\upsilon}}}

\newcommand{\br}{\mathbf{r}}
\newcommand{\bs}{\mathbf{s}}
\def\A{{\mathscr{A}}}
\def\B{{\mathscr{B}}}
\def\C{{\mathscr{C}}}
\def\RR{{\mathscr{R}}}

\def\supp{\hbox{\rm{supp}}}
\def\var{\hbox{\rm{var}}}

\def\Tr{\hbox{\rm{Tr}}}

\def\int{\hbox{\rm{int}}}

\newcommand{\vge}{\mathbin{\rotatebox[origin=c]{90}{$\ge$}}}
\newif\ifcomment
\commentfalse
\commenttrue
\usepackage{todonotes}

\newcommand{\revision}[1]{{{\color{black}#1}}}
\begin{document}


\title[Combining correlative and term sparsity for large-scale POPs]{CS-TSSOS: Correlative and term sparsity for large-scale polynomial optimization}
\author[J. Wang \and V. Magron \and J.B. Lasserre \and N. H. A. Mai]{Jie Wang \and Victor Magron \and Jean B. Lasserre \and Ngoc Hoang Anh Mai}
\if{

}\fi

\date{\today}

\begin{abstract}
This work proposes a new moment-SOS hierarchy, called \emph{CS-TSSOS}, for solving large-scale sparse polynomial optimization problems. Its novelty is to exploit simultaneously \emph{correlative sparsity} and \emph{term sparsity} 
by combining advantages of two existing frameworks for sparse polynomial optimization.
The former is due to Waki et al. \cite{waki} while the latter was initially proposed by Wang et al. \cite{wang} and later exploited in the TSSOS hierarchy \cite{wang3,wang2}.
In doing so we obtain CS-TSSOS -- a two-level hierarchy of semidefinite programming relaxations with 
(i), the crucial property to involve blocks of SDP matrices and (ii), the guarantee of convergence to the global optimum under certain conditions.
We demonstrate its efficiency and scalability on several large-scale instances of the celebrated Max-Cut problem and the important industrial optimal power flow problem, involving up to six thousand variables and tens of thousands of constraints.
\end{abstract}

\maketitle

\section{Introduction}
This paper is concerned with solving large-scale polynomial optimization problems. As is often the case, the polynomials in the problem description involve only a few monomials of low degree and the ultimate goal is to exploit this crucial feature to provide 
semidefinite relaxations that are computationally much cheaper than those of the standard SOS-based hierarchy \cite{Las01} or its sparse version \cite{Las06,waki} based on correlative sparsity.

Throughout the paper, we consider large-scale instances of the following polynomial optimization problem (POP):
\begin{equation}\label{sec2-eq9}
(\textrm{Q}):\quad \rho^*=\inf_{\x}\,\{\,f(\x) : \x\in\mathbf{K}\,\},
\end{equation}
where the objective function $f$ is assumed to be a polynomial in $n$ variables $\x= (x_1,\ldots,x_n)$ and the feasible set $\mathbf{K}\subseteq\R^{n}$ is assumed to be defined by a finite conjunction of $m$ polynomial inequalities, namely
\begin{equation}\label{sec2-eq10}
\mathbf{K} := \{\x\in\R^{n} : g_1(\x)\ge 0, \dots, g_m(\x) \geq 0 \},
\end{equation}
for some polynomials $g_1,\dots, g_m$ in $\x$.
Here ``large-scale'' means that the magnitude of the number of variables $n$ and the number of inequalities $m$ can be both proportional to several thousands. 
A nowadays well-established scheme to handle $(\textrm{Q})$ is the \emph{moment-SOS hierarchy}  \cite{Las01}, where SOS is the abbreviation of \emph{sum of squares}. 
The moment-SOS hierarchy provides a sequence of semidefinite programming (SDP) relaxations, whose optimal values are non-decreasing lower bounds of the global optimum $\rho^*$ of $(\textrm{Q})$. 
Under some mild assumption slightly stronger than compactness, the sequence generically converges to the global optimum in finitely many steps \cite{nie2014optimality}. 
SDP solvers \cite{wolkowicz2012handbook} address a specific class of convex optimization problems, with linear cost and linear matrix inequalities. 
With a priory fixed precision, an SDP can be solved in polynomial time with respect to its input size.
Modern SDP solvers via the interior-point method (e.g.~{\tt Mosek} \cite{andersen2000mosek}) can solve an SDP problem involving matrices of moderate size (say, $\le5,000$) and equality constraints of moderate number (say, $\le20,000$) in reasonable time on a standard laptop \cite{toh}.
The SDP relaxations arising from the moment-SOS hierarchy typically involve matrices of size $\binom{n+d}{d}$ and equality constraints of number $\binom{n+2d}{2d}$, where $d$ is the relaxation order. 
For problems with $n \simeq 200$, it is thus possible to compute the first-order SDP relaxation of a quadratically constrained quadratic problem (QCQP), as one can take $d=1$, yielding $\binom{n+d}{d} \simeq 200$ and $\binom{n+2d}{2d} \simeq 20,000$ (in this case, this relaxation is also known as Shor's relaxation \cite{shor1987quadratic}).
However, the quality of the resulting approximation is often not satisfactory and it is then required to go beyond the first-order relaxation. But for solving the second-order relaxation ($d=2$) one is limited to problems of small size, typically with $\binom{n+4}{4}\le 20,000$ (hence with $n\le 24$) on a standard laptop.
Therefore, in view of the current state of SDP solvers, the dense moment-SOS hierarchy does not scale well enough.

One possible remedy is to rely on alternative weaker positivity certificates, such as the hierarchy of linear programming (LP) relaxations based on Krivine-Stengle's certificates \cite{krivineanneaux, stengle1974nullstellensatz,lasserre2017bounded} or the second-order cone programming (SOCP) relaxation based on (scaled) diagonally dominant sums of squares (DSOS/SDSOS) \cite{ahmadi2019dsos}
to approximate/bound from below the optimum of $(\textrm{Q})$. 
Even though modern LP/SOCP solvers can handle much larger problems by comparison with SDP solvers, they have been shown to provide less accurate bounds, in particular for combinatorial problems \cite{laurent2003comparison}, and do not have the property of finite convergence for continuous problems (not even for convex QCQP problems \cite[Section 9.3]{lasserre-cup2015}). Another important methodology is to reduce the size of SDPs arising in the moment-SOS hierarchy via exploiting structure of POPs.

~

\paragraph{\textbf{Related work for unconstrained POPs}}
A first option is to exploit \emph{term sparsity} for sparse unconstrained problems, i.e.~when $\textbf{K} = \R^n$, $f$ involves a few terms (monomials).
The algorithm consists of automatically reducing the size of the corresponding SDP matrix by eliminating the monomial terms which never appear among the support of SOS decompositions \cite{re2}.
Other classes of positivity certificates have been recently developed with a specific focus on sparse unconstrained problems. 
Instead of trying to decompose a positive polynomial as an SOS, one can try to decompose it as a sum of nonnegative circuits (SONC), by solving a geometric program \cite{iliman2016amoebas} or a second-order cone program \cite{averkov2019optimal,wang2019second}, or alternatively as a sum of arithmetic-geometric-mean-exponentials (SAGE) \cite{chandrasekaran2016relative} by solving a relative entropy program.
Despite their potential efficiency on certain sub-classes of POPs (e.g., sparse POPs with a small number of variables and a high degree), these methods share the common drawback of not providing systematic guarantees of convergence for constrained problems. 
%

~

\paragraph{\textbf{Related work on correlative sparsity}}
In order to reduce the computational burden associated with the dense moment-SOS hierarchy while keeping its nice convergence properties, one possibility is to take into account the sparsity pattern satisfied by the variables of the POP \cite{Las06,waki}.  
The resulting algorithm has been implemented in the {\tt SparsePOP} solver \cite{waki2008algorithm} and can handle sparse problems with up to several hundred variables.
Many applications of interest have been successfully handled thanks to this framework, for instance certified roundoff error bounds in computer arithmetics \cite{toms17,toms18} with up to several hundred variables and constraints, optimal power flow problems \cite{josz2018lasserre} (where a multi-ordered Lasserre hierarchy was proposed) with up to several thousand variables and constraints. 
More recent extensions have been developed for volume computation of sparse semialgebraic sets \cite{tacchi2019exploiting}, approximating regions of attraction of sparse polynomial systems \cite{tacchi2019approximating}, noncommutative POPs \cite{klep2019sparse}, Lipschitz constant estimation of deep networks \cite{chen2020polynomial} and for sparse positive definite functions \cite{mai2020sparse}.
In these applications, the cost polynomial and the constraint polynomials possess a specific \textit{correlative sparsity pattern}.
The resulting sparse moment-SOS hierarchy is obtained by building blocks of SDP matrices with respect to some subsets or \emph{cliques} of the input variables.
When the sizes of these cliques are reasonably small, one can expect to handle problems with a large number of variables. 
For instance, the maximal size of cliques is less than $10$ for some unconstrained problems in \cite{waki} or roundoff error problems in \cite{toms17}, and is less than $20$ for the optimal power flow problems handled in \cite{josz2018lasserre}. 
Even though correlative sparsity has been successfully used to tackle several interesting applications, there are still many POPs that cannot be handled by considering merely correlative sparsity. For instance, there are POPs for which the correlative sparsity pattern is (nearly) dense or which admits a correlative sparsity pattern with variable cliques of large cardinality (say, $>20$), yielding untractable SDPs.

~

\paragraph{\textbf{Related work on term sparsity}}

To overcome these issues, one can exploit 
\emph{term sparsity} as described in \cite{wang,wang2,wang3}.
The \emph{TSSOS hierarchy} from \cite{wang2} as well as the complementary \emph{Chordal-TSSOS} from \cite{wang3} offers some alternative to problems for which the correlative sparsity pattern is dense or nearly dense. 
In both TSSOS and Chordal-TSSOS frameworks a so-called {\em term sparsity pattern (tsp) graph} is associated with the POP. 
The nodes of this tsp graph are monomials (from a monomial basis) needed to construct SOS relaxations of the POP. 
Two nodes are connected via an edge whenever the product of the corresponding monomials appears in the supports of polynomials involved in the POP or is a monomial square.
Note that this graph differs from the \emph{correlative sparsity pattern (csp) graph} used in \cite{waki} where the nodes are the input variables and the edges connect two nodes whenever the corresponding variables appear in the same term of the objective function or in the same constraint.
A two-step iterative algorithm takes as input the tsp graph and enlarges it to exploit the term sparsity in (\textrm{Q}). 
Each iteration consists of two successive operations: 
(i) a support extension operation and (ii) 
either a block closure operation on adjacency matrices in the case of TSSOS \cite{wang2} or a chordal extension operation in the case of Chordal-TSSOS \cite{wang3}.
In doing so one obtains a two-level moment-SOS hierarchy with blocks of SDP matrices. If the sizes of blocks are relatively small then the resulting SDP relaxations become more tractable as their computational cost is significantly reduced.
Another interesting feature of TSSOS is that the block structure obtained at the end of the iterative algorithm automatically induces a partition of the monomial basis, which can be interpreted in terms of sign symmetries of the initial POP. 
TSSOS and Chordal-TSSOS allow one to solve POPs with several hundred variables for which there is no or little correlative sparsity to exploit; see \cite{wang2,wang3} for numerous numerical examples.
One can also rely on symmetry exploitation as in \cite{riener2013exploiting} but this requires quite strong assumptions on the input data, such as invariance of each polynomial $f, g_1,\dots,g_m$ under the action of a finite group.

To tackle large-scale POPs, a natural idea is to simultaneously benefit from correlative and term sparsity patterns. This is the spirit of our contribution. Also in the same vein the work in \cite{miller2019decomposed} combines the (S)DSOS framework \cite{ahmadi2019dsos} with the TSSOS hierarchy but does not provide systematic convergence guarantees.

~

\paragraph{\textbf{Contribution}} 

Our main contribution is as follows:\\
%

$\bullet$ For large-scale POPs with a correlative sparsity pattern, we first apply the usual sparse polynomial optimization framework~\cite{Las06,waki} to get a coarse decomposition in terms of cliques of variables.
Next we apply the term sparsity strategy (either TSSOS or Chordal-TSSOS) to each subsystem (which involves 
only one clique of variables) to reduce the size of SDPs even further. While the overall strategy is quite clear and simple, its implementation
is not trivial and needs some care. Indeed for its coherency one needs to take extra care of the monomials 
which involve variables that belong to intersections of variable cliques (those obtained from correlative sparsity). 
The resulting combination of correlative sparsity (CS for short) and term sparsity produces what we call the \emph{CS-TSSOS} hierarchy -- a two-level hierarchy of SDP relaxations with blocks of SDP matrices, which yields a converging sequence of certified approximations for POPs.
Under certain conditions, we prove that the corresponding sequence of optimal values converges to the global optimum of the POP.
%

$\bullet$ Our algorithmic development of the CS-TSSOS hierarchy is fully implemented in the \texttt{TSSOS} tool \cite{magron2021tssos}. 
The most recent version of \texttt{TSSOS} has been released within the Julia programming language, which is freely available online and documented.\footnote{\url{https://github.com/wangjie212/TSSOS}}
With \texttt{TSSOS}, the accuracy and scalability of the CS-TSSOS hierarchy are evaluated on several large-scale benchmarks coming from the continuous and combinatorial optimization literature. 
In particular, numerical experiments demonstrate that the CS-TSSOS hierarchy is able to handle challenging Max-Cut instances and optimal power flow instances with several thousand ($\simeq6,000$) variables on a laptop whenever appropriate sparsity patterns are accessible. 
We remark that the CS-TSSOS framework has been recently extended to handle noncommutative polynomial optimization \cite{wang2020exploiting} and complex polynomial optimization \cite{wang2021exploiting}.
%

The rest of the paper is organized as follows: in Section~\ref{sec:prelim}, we provide preliminary background on SOS polynomials, the moment-SOS hierarchy, correlative sparsity and the (Chordal-)TSSOS hierarchy. 
In Section~\ref{sec:cstssos}, we explain 
how to combine correlative sparsity and term sparsity
to obtain a two-level CS-TSSOS hierarchy. 
Its convergence is analyzed in Section~\ref{sec:cvg}. 
Eventually, we provide numerical experiments for large-scale POP instances in Section~\ref{sec:benchs}. Discussions and conclusions are made in Section~\ref{conc}.

\section{Notation and Preliminaries}
\label{sec:prelim}
\subsection{Notation and SOS polynomials}\label{SOS}
Let $\x=(x_1,\ldots,x_n)$ be a tuple of variables and $\R[\x]=\R[x_1,\ldots,x_n]$ be the ring of real $n$-variate polynomials. For $d\in\N$, the set of polynomials of degree no more than $2d$ is denoted by $\R_{2d}[\x]$. A polynomial $f\in\R[\x]$ can be written as $f(\x)=\sum_{\a\in\A}f_{\a}\x^{\a}$ with $\A\subseteq\N^n$ and $f_{\a}\in\R, \x^{\a}=x_1^{\alpha_1}\cdots x_n^{\alpha_n}$. The support of $f$ is defined by $\supp(f):=\{\a\in\A\mid f_{\a}\ne0\}$. We use $|\cdot|$ to denote the cardinality of a set. For a finite set $\A\subseteq\N^n$, let $\x^{\mathscr{A}}$ be the $|\mathscr{A}|$-dimensional column vector consisting of elements $\x^{\a},\a\in\mathscr{A}$ (fix any ordering on $\N^n$). For a positive integer $r$, the set of $r\times r$ symmetric matrices is denoted by $\mathbf{S}^r$ and the set of $r\times r$ positive semidefinite (PSD) matrices is denoted by $\mathbf{S}_+^r$. \revision{A matrix $A\in\mathbf{S}_+^r$ is written as $A\succeq0$. For matrices $A,B\in\mathbf{S}^r$, let $\langle A, B\rangle\in\R$ denote the trace inner-product, defined by $\langle A, B\rangle=\Tr(A^TB)$, and let $A\circ B\in\mathbf{S}^r$ denote the Hadamard product, defined by $[A\circ B]_{ij} = A_{ij}B_{ij}$. For $d\in\N$, let $\N^n_d:=\{\a=(\alpha_i)_{i=1}^n\in\N^n\mid\sum_{i=1}^n\alpha_i\le d\}$.
For $\b=(\beta_i)\in\N^n,\g=(\gamma_i)\in\N^n$, let $\b+\g:=(\beta_i+\gamma_i)\in\N^n$. For $\a\in\N^n,\A,\B\subseteq\N^n$,
let $\a+\B:=\{\a+\b\mid\b\in\B\}$ and $\A+\B:=\{\a+\b\mid\a\in\A,\b\in\B\}$. For $m\in\N\backslash\{0\}$, let $[m]:=\{1,2,\ldots,m\}$.}

Given a polynomial $f(\x)\in\R[\x]$, if there exist polynomials $f_1(\x),\ldots,f_t(\x)$ such that $f(\x)=\sum_{i=1}^tf_i(\x)^2$,
then we call $f(\x)$ a {\em sum of squares (SOS)} polynomial. The set of SOS polynomials is denoted by $\Sigma[\x]$. Assume that $f\in\Sigma_{2d}[\x]:=\Sigma[\x]\cap\R_{2d}[\x]$ and $\x^{\N^n_{d}}$ is the {\em standard monomial basis}. Then the SOS condition for $f$ is equivalent to the existence of a PSD matrix $Q$, which is called a {\em Gram matrix} \cite{re2}, such that $f=(\x^{\N^n_{d}})^TQ\x^{\N^n_{d}}$.
For convenience, we abuse notation in the sequel and denote by $\N^n_{d}$ instead of $\x^{\N^n_{d}}$ the standard monomial basis and use the exponent $\a$ to represent a monomial $\x^{\a}$.

\subsection{The moment-SOS hierarchy for POPs}\label{moment-SOS}
With $\y=(y_{\a})_{\a}$ being a sequence indexed by the standard monomial basis $\N^n$ of $\R[\x]$, let $L_{\y}:\R[\x]\rightarrow\R$ be the linear functional
\begin{equation*}
f=\sum_{\a}f_{\a}\x^{\a}\mapsto L_{\y}(f)=\sum_{\a}f_{\a}y_{\a}.
\end{equation*}
For $d\in\N$, the {\em moment matrix} $M_{d}(\y)$ of order $d$ associated with $\y$ is the matrix with rows and columns indexed by the standard monomial basis $\N^n_{d}$ such that
\begin{equation*}
M_d(\y)_{\b\g}:=L_{\y}(\x^{\b}\x^{\g})=y_{\b+\g}, \quad\forall\b,\g\in\N^n_{d}.
\end{equation*}

Suppose $g=\sum_{\a}g_{\a}\x^{\a}\in\R[\x]$ and let $\y=(y_{\a})$ be given. The {\em localizing matrix} $M_{d}(g\y)$ of order $d$ associated with $g$ and $\y$ is the matrix with rows and columns indexed by $\N^n_{d}$ such that
\begin{equation*}
M_{d}(g\,\y)_{\b\g}:=L_{\y}(g\,\x^{\b}\x^{\g})=\sum_{\a}g_{\a}y_{\a+\b+\g}, \quad\forall\b,\g\in\N^n_{d}.
\end{equation*}

Consider the POP ($\textrm{Q}$) defined by \eqref{sec2-eq9} and \eqref{sec2-eq10}. Throughout the paper let $d_j:=\lceil\deg(g_j)/2\rceil,j=1,\ldots,m$ and $d_{\min}:=\max\{\lceil\deg(f)/2\rceil,d_1,\ldots,d_m\}$. Then the moment hierarchy for ($\textrm{Q}$) indexed by integer $d\ge d_{\min}$ is defined by (\cite{Las01}):
\begin{equation}\label{sec2-eq11}
(\textrm{Q}_{d}):\quad
\begin{cases}
\inf& L_{\y}(f)\\
\textrm{s.t.}&M_{d}(\y)\succeq0,\\
&M_{d-d_j}(g_j\y)\succeq0,\quad j=1,\ldots,m,\\
&y_{\mathbf{0}}=1.
\end{cases}
\end{equation}
We call $d$ the {\em relaxation order}. 

For the sake of convenience, we set $g_0:=1$ and $d_0:=0$ throughout the paper. For each $j$, writing $M_{d-d_j}(g_j\y)=\sum_{\a}D_{\a}^jy_{\a}$ for appropriate symmetry matrices $\{D_{\a}^j\}$, the dual of \eqref{sec2-eq11} reads as
\begin{equation}\label{sec2-eq12}
(\textrm{Q}_{d})^*:\quad
\begin{cases}
\sup&\rho\\
\textrm{s.t.}&\displaystyle\sum_{j=0}^m\langle Q_j,D_{\a}^j\rangle+\rho\delta_{\mathbf{0}\a}=f_{\a},\quad\forall\a\in\N^n_{2d},\\
&Q_j\succeq0,\quad j=0,\ldots,m,
\end{cases}
\end{equation}
where $\delta_{\mathbf{0}\a}$ is the usual Kronecker symbol.

\subsection{Chordal graphs and sparse matrices}
In this subsection, we recall some basic results on chordal graphs and sparse matrices which are crucial for our subsequent development. 

An {\em (undirected) graph} $G(V,E)$ or simply $G$ consists of a set of nodes $V$ and a set of edges $E\subseteq\{\{v_i,v_j\}\mid v_i\ne v_j,(v_i,v_j)\in V\times V\}$. For a graph $G$, we use $V(G)$ and $E(G)$ to indicate the node set of $G$ and the edge set of $G$, respectively. \revision{The {\em adjacency matrix} of a graph $G$ is denoted by $B_G$ for which we put ones on its diagonal.} For two graphs $G,H$, we say that $G$ is a {\em subgraph} of $H$, denoted by $G\subseteq H$, if both $V(G)\subseteq V(H)$ and $E(G)\subseteq E(H)$ hold.
\begin{definition}
A graph is called a {\em chordal graph} if all its cycles of length at least four have a chord\footnote{A chord is an edge that joins two nonconsecutive nodes in a cycle.}.
\end{definition}
The notion of chordal graphs plays an important role in sparse matrix theory. Any non-chordal graph $G(V,E)$ can be always extended to a chordal graph $G'(V,E')$ by adding appropriate edges to $E$, which is called a {\em chordal extension} of $G(V,E)$. \revision{As an example, in Figure \ref{ce} the two dashed edges are added to obtain a chordal extension.} The chordal extension of $G$ is usually not unique and the symbol $G'$ is used to represent any specific chordal extension of $G$ throughout the paper. For graphs $G\subseteq H$, we assume that $G'\subseteq H'$ always holds in this paper. 

\revision{\begin{figure}[htbp]
\caption{An example of chordal extension}\label{ce}
\begin{center}
{\tiny
\begin{tikzpicture}[every node/.style={circle, draw=blue!50, thick, minimum size=8mm}]
\node (n1) at (-2,0) {$1$};
\node (n2) at (0,0) {$2$};
\node (n3) at (2,0) {$3$};
\node (n4) at (-2,-2) {$4$};
\node (n5) at (0,-2) {$5$};
\node (n6) at (2,-2) {$6$};
\draw (n1)--(n4);
\draw (n1)--(n2);
\draw (n2)--(n3);
\draw (n2)--(n5);
\draw (n3)--(n6);
\draw (n4)--(n5);
\draw (n5)--(n6);
\draw[dashed] (n1)--(n5);
\draw[dashed] (n2)--(n6);
\end{tikzpicture}}
\end{center}
\end{figure}
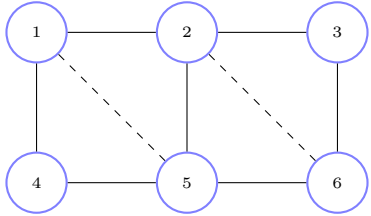}

\revision{A {\em complete graph} is a graph in which any two nodes have an edge.
A {\em clique} of a graph is a subset of nodes that induces a complete subgraph. A {\em maximal clique} is a clique that is not contained in any other clique.} It is known that for a chordal graph, its maximal cliques can be enumerated efficiently in linear time in terms of the number of nodes and edges. See e.g.\,\cite{bp,fg,go} for the details.

From now on we consider graphs with the node set $V\subseteq\N^n$.
Given a graph $G(V,E)$, a symmetric matrix $Q$ with rows and columns indexed by $V$ is said to have sparsity pattern $G$ if $Q_{\b\g}=Q_{\g\b}=0$ whenever $\b\ne\g$ and $\{\b,\g\}\notin E$. Let $\mathbf{S}_G$ be the set of symmetric matrices with sparsity pattern $G$. \revision{Matrices in $\mathbf{S}_G$ possess a {\em block structure}: each block is indexed by a maximal clique of $G$. The maximal size of blocks is the maximal size of maximal cliques of $G$, namely, the \emph{clique number} of $G$.}

\begin{remark}\label{rm-max}
For a graph $G$, among all chordal extensions of $G$, there is a particular one $G'$ which makes every connected component of $G$ to be a complete subgraph. \revision{Accordingly, the matrix with sparsity pattern $G'$ is block diagonal (after an appropriate permutation on rows and columns): each block corresponds to a connected component of $G$}. We call this chordal extension the {\em maximal} chordal extension.
In this paper, we only consider chordal extensions that are subgraphs of the maximal chordal extension. 
\end{remark}

Given a graph $G(V,E)$, the PSD matrices with sparsity pattern $G$ form a convex cone
\begin{equation}\label{sec2-eq5}
\mathbf{S}_+^{|V|}\cap\mathbf{S}_G=\{Q\in\mathbf{S}_G\mid Q\succeq0\}.
\end{equation}
Once the sparsity pattern graph $G(V,E)$ is a chordal graph, the cone $\mathbf{S}_+^{|V|}\cap\mathbf{S}_G$ can be
decomposed as a sum of simple convex cones thanks to the following theorem and hence the related optimization problem can be solved more efficiently.
\begin{theorem}[\cite{agler}, Theorem 2.3]\label{sec2-thm}
Let $G(V,E)$ be a chordal graph and assume that $C_1,\ldots,C_t$ are the list of maximal cliques of $G(V,E)$. Then a matrix $Q\in\mathbf{S}_+^{|V|}\cap\mathbf{S}_G$ if and only if $Q$ can be written as $Q=\sum_{i=1}^tQ_{i}$, where $Q_i\in\mathbf{S}_+^{|V|}$ has nonzero entries only with row and column indices coming from $C_i$ for $i=1,\ldots,t$.
\end{theorem}

Given a graph $G(V,E)$, let $\Pi_{G}$ be the projection from $\mathbf{S}^{|V|}$ to the subspace $\mathbf{S}_G$, i.e., for $Q\in\mathbf{S}^{|V|}$,
\begin{equation}\label{sec2-eq7}
\Pi_{G}(Q)_{\b\g}=\begin{cases}
Q_{\b\g}, &\textrm{if }\b=\g\textrm{ or }\{\b,\g\}\in E,\\
0, &\textrm{otherwise}.
\end{cases}
\end{equation}
\revision{The set $\Pi_{G}(\mathbf{S}_+^{|V|})$ denotes matrices in $\mathbf{S}_G$ that have a PSD completion in the sense that
diagonal entries and off-diagonal entries corresponding to edges of $G$ are fixed; other off-diagonal entries are free}. More precisely, 
\begin{equation}\label{sec2-eq8}
\Pi_{G}(\mathbf{S}_+^{|V|})=\{\Pi_{G}(Q)\mid Q\in\mathbf{S}_+^{|V|}\}.
\end{equation}
One can easily check that the PSD completable cone $\Pi_{G}(\mathbf{S}_+^{|V|})$ and the PSD cone $\mathbf{S}_+^{|V|}\cap\mathbf{S}_G$ form a pair of dual cones in $\mathbf{S}_G$. Moreover, for a chordal graph $G$, the decomposition result for matrices in $\mathbf{S}_+^{|V|}\cap\mathbf{S}_G$ given in Theorem \ref{sec2-thm} leads to the following characterization of matrices in the PSD completable cone $\Pi_{G}(\mathbf{S}_+^{|V|})$.
\begin{theorem}[\cite{grone1984}, Theorem 7]\label{sec2-thm2}
Let $G(V,E)$ be a chordal graph and assume that $C_1,\ldots,C_t$ are the list of maximal cliques of $G(V,E)$. Then a matrix $Q\in\Pi_{G}(\mathbf{S}_+^{|V|})$ if and only if $Q[C_i]\succeq0$ for $i=1,\ldots,t$, where $Q[C_i]$ denotes the principal submatrix of $Q$ indexed by the clique $C_i$.
\end{theorem}

\revision{By Theorem \ref{sec2-thm2}, to check $Q\in\Pi_{G}(\mathbf{S}_+^{|V|})$, it suffices to check the positive semidefiniteness of certain blocks of $Q$.} For more details on chordal graphs and sparse matrices, the reader may refer to \cite{va}.

\subsection{Correlative sparsity}\label{cs}
To exploit correlative sparsity in the moment-SOS hierarchy for POPs, one proceeds
in two steps: 1) partition the set of variables into cliques according to the links between variables emerging in the input polynomial system, and 2) construct a sparse moment-SOS hierarchy with respect to the former partition of variables \cite{waki}.

More concretely, we define the {\em correlative sparsity pattern (csp) graph} associated with POP \eqref{sec2-eq9} to be the graph $G^{\textrm{csp}}$ with nodes $V=[n]$ and edges $E$ satisfying $\{i,j\}\in E$ if one of following holds:
\begin{enumerate}
    \item[(i)] there exists $\a\in\supp(f)\textrm{ s.t. }\alpha_i>0,\alpha_j>0$;
    \item[(ii)] there exists $k\in[m]$ such that $x_i,x_j\in\var(g_k)$, where $\var(g_k)$ is the set of variables involved in $g_k$.
\end{enumerate}
Let $(G^{\textrm{csp}})'$ be a chordal extension of $G^{\textrm{csp}}$ and $\{I_l\}_{l=1}^p$ be the list of maximal cliques of $(G^{\textrm{csp}})'$ with $n_l:=|I_l|$. Let $\R[\x(I_l)]$ denote the ring of polynomials in the $n_l$ variables $\x(I_l) = \{x_i\mid i\in I_l\}$. We then partition the constraint polynomials $g_1,\ldots,g_m$ into groups $\{g_j\mid j\in J_l\}, l=1,\ldots,p$ which satisfy:
\begin{enumerate}
    \item[(i)] $J_1,\ldots,J_p\subseteq[m]$ are pairwise disjoint and $\cup_{l=1}^pJ_l=[m]$;
    \item[(ii)] for any $j\in J_l$, $\var(g_j)\subseteq I_l$, $l=1,\ldots,p$.
\end{enumerate}

Next, with $l\in\{1,\ldots, p\}$ fixed, for $d\in\N$ and $g\in\R[\x(I_l)]$, let $M_d(\y, I_l)$ (resp. $M_d(g\y, I_l)$)
be the moment (resp. localizing) submatrix obtained from $M_d(\y)$ (resp. $M_d(g\y)$) by retaining only those rows and columns indexed by $\b=(\beta_i)\in\N_d^n$ of $M_d(\y)$ (resp. $M_d(g\y)$) with $\supp(\b)\subseteq I_l$, where $\supp(\b):=\{i\mid \beta_i\ne0\}$.

Then with $d\ge d_{min}$, the moment hierarchy based on correlative sparsity for POP \eqref{sec2-eq9} is defined as:
\begin{equation}\label{sec2-eq13}
(\textrm{Q}^{\textrm{cs}}_{d}):\quad
\begin{cases}
\inf &L_{\y}(f)\\
\textrm{s.t.}&M_{d}(\y, I_l)\succeq0,\quad l=1,\ldots,p,\\
&M_{d-d_j}(g_j\y, I_l)\succeq0,\quad j\in J_l, l=1\,\ldots,p,\\
&y_{\mathbf{0}}=1,
\end{cases}
\end{equation}
with optimal value denoted by $\rho_{d}$. In the following, we refer to $(\textrm{Q}^{\textrm{cs}}_{d})$ \eqref{sec2-eq13} as the {\em CSSOS} hierarchy for POP \eqref{sec2-eq9}.

\begin{remark}
As shown in \cite{Las06} under some compactness assumption, the sequence $(\rho_{d})_{d\ge d_{min}}$ monotonically converges to the global optimum $\rho^*$ of POP \eqref{sec2-eq9}.
\end{remark}

\subsection{Term sparsity}\label{ts}
In contrast to the correlative sparsity pattern which focuses on links between \emph{variables}, the term sparsity pattern focuses on links between \emph{monomials} (or terms). To
exploit term sparsity in the moment-SOS hierarchy one also proceeds in two steps: 1) partition each involved monomial basis into blocks according to the links between monomials emerging in the input polynomial system, and 2) construct a sparse moment-SOS hierarchy with respect to the former partition of monomial bases \cite{wang2,wang3}.

More concretely, let
$\mathscr{A} = \supp(f)\cup\bigcup_{j=1}^m\supp(g_j)$ and $\N^n_{d-d_j}$ be the standard monomial basis for $j=0,\ldots,m$. Fixing a relaxation order $d\ge d_{min}$, we define the {\em term sparsity pattern (tsp) graph} associated with POP \eqref{sec2-eq9} or the support set $\A$, to be the graph $G_{d}^{\textrm{tsp}}$ with node set $V_{d,0}:=\N^n_{d}$ and edge set
\begin{equation}\label{sec2-eq14}
E:=\{\{\b,\g\}\mid\b\ne\g\in V, \b+\g\in\A\cup(2\N)^n\},
\end{equation}
\revision{where $(2\N)^n:=\{2\a\mid\a\in\N^n\}.$

For a graph $G(V,E)$ with $V\subseteq\N^n$, let $\supp(G):=\{\b+\g\mid\{\b,\g\}\in E\}$.
Assume that $G_{d,0}^{(0)}=G_{d}^{\textrm{tsp}}$ and $G_{d,j}^{(0)}$ with node set $V_{d,j}:=\N^n_{d-d_j}$ is an empty graph (i.e., with empty edge set) for $j=1,\ldots,m$.
Now for each $j\in\{0\}\cup[m]$, we iteratively define an ascending chain of graphs $(G_{d,j}^{(k)}(V_{d,j},E_{d,j}^{(k)}))_{k\ge1}$. To this end, we start with the initial graph $G_{d,j}^{(0)}$ and each iteration consists of two successive operations:\\
1) {\bf support extension}: Define $F_{d,j}^{(k)}$ to be the graph with nodes $V_{d,j}$ and with
\begin{align}\label{sec2-eq15}
E(F_{d,j}^{(k)})=&\{\{\b,\g\}\mid\b\ne\g\in V_{d,j},\\
&(\b+\g+\supp(g_j))\cap(\cup_{i=0}^m\supp(G_{d,i}^{(k-1)}))\ne\emptyset\},\quad j\in\{0\}\cup[m].\notag
\end{align}
2) {\bf chordal extension}: Let
\begin{equation}\label{sec2-graph2}
G_{d,j}^{(k)}:=(F_{d,j}^{(k)})',\quad j\in\{0\}\cup[m].
\end{equation}
To summarise, the iterative process is
\begin{equation*}
G_{d,j}^{(0)}\rightarrow\cdots\rightarrow G_{d,j}^{(k-1)}\xrightarrow{\textrm{support extension}}F_{d,j}^{(k)}\xrightarrow{\textrm{chordal extension}}G_{d,j}^{(k)}\rightarrow\cdots,
\end{equation*}
for each $j\in\{0\}\cup[m]$.

\begin{example}[support extension]
Assume $m=0$ and consider the graph $G$ with solid edges shown in Figure \ref{ts-ex}. Then by support extension, the two dashed edges are added to $G$ for $x_1x_2x_3\in\supp(G)$.

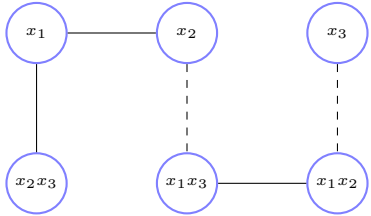
\begin{figure}[htbp]
\caption{The support extension of $G$}\label{ts-ex}
\begin{center}
{\tiny
\begin{tikzpicture}[every node/.style={circle, draw=blue!50, thick, minimum size=8mm}]
\node (n1) at (-2,0) {$x_1$};
\node (n2) at (0,0) {$x_2$};
\node (n3) at (2,0) {$x_3$};
\node (n4) at (-2,-2) {$x_2x_3$};
\node (n5) at (0,-2) {$x_1x_3$};
\node (n6) at (2,-2) {$x_1x_2$};
\draw (n1)--(n4);
\draw[dashed] (n2)--(n5);
\draw[dashed] (n3)--(n6);
\draw (n1)--(n2);
\draw (n5)--(n6);
\end{tikzpicture}}
\end{center}
\end{figure}

\end{example}}

Let $r_j:=|\N^n_{d-d_j}|=\binom{n+d-d_j}{d-d_j},j=0,\ldots,m$. Then with $d\ge d_{min}$ and $k\ge1$, the moment hierarchy based on term sparsity for POP \eqref{sec2-eq9} is defined as:
\begin{equation}\label{sec2-eq16}
(\textrm{Q}^{\textrm{ts}}_{d,k}):\quad
\begin{cases}
\inf &L_{\y}(f)\\
\textrm{s.t.}&B_{G_{d,0}^{(k)}}\circ M_{d}(\y)\in\Pi_{G_{d,0}^{(k)}}(\mathbf{S}_+^{r_0}),\\
&B_{G_{d,j}^{(k)}}\circ M_{d-d_j}(g_j\y)\in\Pi_{G_{d,j}^{(k)}}(\mathbf{S}_+^{r_j}),\quad j=1,\ldots,m,\\
&y_{\mathbf{0}}=1.
\end{cases}
\end{equation}
We call $k$ the {\em sparse order} and in the remainder of this paper, the {\em TSSOS} hierarchy for POP \eqref{sec2-eq9}
refers to the hierarchy $(\textrm{Q}^{\textrm{ts}}_{d,k})$.

\begin{remark}
In $(\emph{Q}^{\emph{ts}}_{d,k})$, one has the freedom to choose a specific chordal extension for any involved graph $G_{d,j}^{(k)}$. As shown in \cite{wang2}, if one chooses the maximal chordal extension then with $d$ fixed, the resulting sequence of optimal values of the TSSOS hierarchy (as $k$ increases) monotonically converges in finitely many steps to the optimal value of the corresponding dense moment relaxation for POP \eqref{sec2-eq9}.
\end{remark}


\section{The CS-TSSOS Hierarchy}
\label{sec:cstssos}
When applicable, one can significantly improve the scalability of the moment-SOS hierarchy by exploiting correlative sparsity or term sparsity.
For large-scale POPs, it is then natural to ask whether one can combine correlative sparsity and term sparsity to further reduce the size of SDPs involved in the moment-SOS hierarchy and to improve its scalability even more. As we shall see in the following sections, the answer is affirmative.

\subsection{The CS-TSSOS Hierarchy for general POPs}\label{sec4-1}
Let us continue considering POP \eqref{sec2-eq9}\footnote{
Though we only include inequality constraints in the definition of $\mathbf{K}$ \eqref{sec2-eq10} for the sake of simplicity, equality constraints can be treated in a similar way.}. 
A first natural idea to combine correlative sparsity and term sparsity would be to apply the TSSOS hierarchy for each subsystem (involving one variable clique) \emph{separately}, once the cliques have been obtained from the csp graph of POP \eqref{sec2-eq9}.
However, with this naive approach convergence may be lost and in the following we take extra care to avoid this annoying consequence.

Let $G^{\textrm{csp}}$ be the csp graph associated with POP \eqref{sec2-eq9}, $(G^{\textrm{csp}})'$ a chordal extension of $G^{\textrm{csp}}$ and $\{I_l\}_{l=1}^p$ be the list of maximal cliques of $(G^{\textrm{csp}})'$ with $n_l:=|I_l|$. As in Section~\ref{cs}, the set of variables $\x$ is partitioned into $\x(I_1), \x(I_2), \ldots, \x(I_p)$. Let $J_1,\ldots,J_p$ be defined as in Section~\ref{cs}.

Now we apply the term sparsity pattern to each subsystem involving variables $\x(I_l)$, $l=1,\ldots,p$ respectively as follows. Let
\begin{equation}\label{sec4-eq0}
    \mathscr{A} := \supp(f)\cup\bigcup_{j=1}^m\supp(g_j)\textrm{ and }
    \mathscr{A}_l := \{\a\in\A\mid\supp(\a)\subseteq I_l\}
\end{equation}
for $l=1,\ldots,p$. As before, we set $d_{\min}:= \max\{\lceil\deg(f)/2\rceil,d_1,\ldots,d_m\}$, $d_0:=0$ and $g_0:=1$. Fix a relaxation order $d\ge d_{\min}$ and let $\N^{n_l}_{d-d_j}$ be the standard monomial basis for $j\in\{0\}\cup J_l,l=1\ldots,p$. Let $G_{d,l}^{\textrm{tsp}}$ be the tsp graph with nodes $\N^{n_l}_{d}$ associated with $\A_l$ defined as in Section~\ref{ts}. Note that we embed $\N^{n_l}$ into $\N^{n}$ via the map
$\a=(\alpha_i)\in\N^{n_l}\mapsto\a'=(\alpha'_i)\in\N^{n}$ which satisfies 
\begin{equation*}
    \alpha'_i=\begin{cases}
    \alpha_i,\quad\textrm{if } i\in I_l,\\
    0,\quad\,\,\,\textrm{otherwise. }
    \end{cases}
\end{equation*}

\revision{Let us assume that $G_{d,l,0}^{(0)}=G_{d,l}^{\textrm{tsp}}$ and $G_{d,l,j}^{(0)},j\in J_l, l=1,\ldots,p$ are all empty graphs. Next for each $j\in\{0\}\cup J_l,l=1,\ldots,p$, we iteratively define an ascending chain of graphs $(G_{d,l,j}^{(k)}(V_{d,l,j},E_{d,l,j}^{(k)}))_{k\ge1}$ with $V_{d,l,j}:=\N^{n_l}_{d-d_j}$ via two successive operations:\\
1) {\bf support extension}: Define $F_{d,l,j}^{(k)}$ to be the graph with nodes $V_{d,l,j}$ and with
\begin{equation}\label{sec4-eq1}
E(F_{d,l,j}^{(k)})=\{\{\b,\g\}\mid\b\ne\g\in V_{d,l,j},(\b+\g+\supp(g_j))\cap\C_{d}^{(k-1)}\ne\emptyset\},
\end{equation}
where
\begin{equation}\label{sec4-eq11}
    \C_{d}^{(k-1)}:=\bigcup_{l=1}^p(\cup_{j\in \{0\}\cup J_l}(\supp(g_j)+\supp(G_{d,l,j}^{(k-1)}))).
\end{equation}
2) {\bf chordal extension}: Let
\begin{equation}\label{sec4-graph}
G_{d,l,j}^{(k)}:=(F_{d,l,j}^{(k)})',\quad j\in\{0\}\cup J_l,l=1,\ldots,p.
\end{equation}}

\begin{example}\label{ex1}
Let $f=1+x_1^2+x_2^2+x_3^2+x_1x_2+x_2x_3+x_3$ and consider the unconstrained POP: $\min\{f(\x): \x\in\R^n\}$. We have $n=3,m=0$ and $d=d_{\min}=1$. The variables are partitioned into two cliques: $\{x_1,x_2\}$ and $\{x_2,x_3\}$. The tsp graphs with respect to these two cliques are illustrated in Figure \ref{ex1-1}. The left graph corresponds to the first clique: $x_1$ and $x_2$ are connected because of the term $x_1x_2$. The right graph corresponds to the second clique: $1$ and $x_3$ are connected because of the term $x_3$; $x_2$ and $x_3$ are connected because of the term $x_2x_3$. If we apply the TSSOS hierarchy (using the maximal chordal extension in \eqref{sec4-graph}) separately for each clique, then the graph sequences $(G_{1,l}^{(k)})_{k\ge1},l=1,2$ (the subscript $j$ is omitted here since there is no constraint) stabilize at $k=1$. However, the added (dashed) edge in the right graph corresponds to the monomial $x_2$, which only involves the variable $x_2$ belonging to the first clique. Hence we need to add the edge connecting $1$ and $x_2$ to the left graph in order to get the guarantee of convergence as we shall see in Section~\ref{glo-con}. Consequently, the graph sequences $(G_{1,l}^{(k)})_{k\ge1},l=1,2$ stabilize at $k=2$.

\begin{figure}[htbp]
\caption{The tsp graphs of Example \ref{ex1}. The dashed edge is added after the maximal chordal extension.}\label{ex1-1}
\begin{center}
\begin{minipage}{0.3\linewidth}
{\tiny
\begin{tikzpicture}[every node/.style={circle, draw=blue!50, thick, minimum size=8mm}]
\node (n1) at (90:1) {$1$};
\node (n2) at (330:1) {$x_2$};
\node (n3) at (210:1) {$x_1$};
\draw (n2)--(n3);
\end{tikzpicture}}
\end{minipage}
\begin{minipage}{0.3\linewidth}
{\tiny
\begin{tikzpicture}[every node/.style={circle, draw=blue!50, thick, minimum size=8mm}]
\node (n1) at (90:1) {$1$};
\node (n2) at (330:1) {$x_3$};
\node (n3) at (210:1) {$x_2$};
\draw[dashed] (n1)--(n3);
\draw (n1)--(n2);
\draw (n2)--(n3);
\end{tikzpicture}}
\end{minipage}
\end{center}
\end{figure}
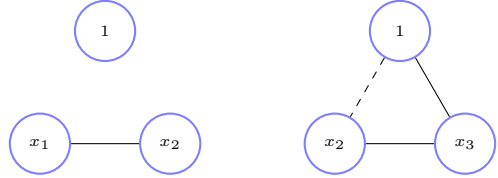
\end{example}

Let $r_{l,j}:=|\N^{n_l}_{d-d_j}|=\binom{n_l+d-d_j}{d-d_j}$ for all $l,j$. Then with $k\ge1$, the moment hierarchy based on correlative-term sparsity for POP \eqref{sec2-eq9} is defined as:
\begin{equation}\label{sec4-eq2}
(\textrm{Q}^{\textrm{cs-ts}}_{d,k}):\quad
\begin{cases}
\inf&L_{\y}(f)\\
\textrm{s.t.}&B_{G_{d,l,0}^{(k)}}\circ M_{d}(\y, I_l)\in\Pi_{G_{d,l,0}^{(k)}}(\mathbf{S}_+^{r_{l,0}}),\quad l=1,\ldots,p,\\
&B_{G_{d,l,j}^{(k)}}\circ M_{d-d_j}(g_j\y, I_l)\in\Pi_{G_{d,l,j}^{(k)}}(\mathbf{S}_+^{r_{l,j}}),\quad j\in J_l,l=1,\ldots,p,\\
&y_{\mathbf{0}}=1,
\end{cases}
\end{equation}
with optimal value denoted by $\rho^{(k)}_{d}$.

\begin{proposition}\label{sec4-prop1}
Fixing a relaxation order $d\ge d_{\min}$, the sequence $(\rho^{(k)}_{d})_{k\ge1}$ is monotonically non-decreasing and $\rho^{(k)}_{d}\le\rho_{d}$ for all $k$.
\end{proposition}
\begin{proof}
By construction, we have $G_{d,l,j}^{(k)}\subseteq G_{d,l,j}^{(k+1)}$ for all $d,l,j$ and all $k$. It follows that each maximal clique of $G_{d,l,j}^{(k)}$ is a subset of some maximal clique of $G_{d,l,j}^{(k+1)}$. Hence by Theorem \ref{sec2-thm2}, $(\textrm{Q}^{\textrm{cs-ts}}_{d,k})$ is a relaxation of $(\textrm{Q}^{\textrm{cs-ts}}_{d,k+1})$ and is clearly also a relaxation of $(\textrm{Q}^{\textrm{cs}}_{d})$. Therefore, $(\rho^{(k)}_{d})_{k\ge1}$ is monotonically non-decreasing and $\rho^{(k)}_{d}\le\rho_{d}$ for all $k$.
\end{proof}

\begin{proposition}\label{sec4-prop2}
Fixing a sparse order $k\ge 1$, the sequence $(\rho^{(k)}_{d})_{d\ge d_{min}}$ is monotonically non-decreasing.
\end{proposition}
\begin{proof}
The conclusion follows if we can show that $G_{d,l,j}^{(k)}\subseteq G_{d+1,l,j}^{(k)}$ for all $d,l,j,k$ since by Theorem \ref{sec2-thm2} this implies that $(\textrm{Q}^{\textrm{cs-ts}}_{d,k})$ is a relaxation of $(\textrm{Q}^{\textrm{cs-ts}}_{d+1,k})$. Let us prove $G_{d,l,j}^{(k)}\subseteq G_{d+1,l,j}^{(k)}$ by induction on $k$. For $k=1$, from $\eqref{sec2-eq14}$, we have $G_{d,l,0}^{(0)}=G_{d,l}^{\textrm{tsp}}\subseteq G_{d+1,l}^{\textrm{tsp}}=G_{d+1,l,0}^{(0)}$, which together with $\eqref{sec4-eq1}$-$\eqref{sec4-eq11}$ implies that $F_{d,l,j}^{(1)}\subseteq F_{d+1,l,j}^{(1)}$ for $j\in \{0\}\cup J_l,l=1,\ldots,p$. It then follows that $G_{d,l,j}^{(1)}=(F_{d,l,j}^{(1)})'\subseteq (F_{d+1,l,j}^{(1)})'=G_{d+1,l,j}^{(1)}$. Now assume that $G_{d,l,j}^{(k)}\subseteq G_{d+1,l,j}^{(k)}$, $j\in \{0\}\cup J_l,l=1,\ldots,p$, holds for some $k\geq 1$. Then by $\eqref{sec4-eq1}$-$\eqref{sec4-eq11}$ and by the induction hypothesis, we have $F_{d,l,j}^{(k+1)}\subseteq F_{d+1,l,j}^{(k+1)}$ for $j\in \{0\}\cup J_l,l=1,\ldots,p$. Thus $G_{d,l,j}^{(k+1)}=(F_{d,l,j}^{(k+1)})'\subseteq (F_{d+1,l,j}^{(k+1)})'=G_{d+1,l,j}^{(k+1)}$ which completes the induction. 
\end{proof}

From Proposition \ref{sec4-prop1} and Proposition \ref{sec4-prop2}, we deduce the following two-level hierarchy of lower bounds for the optimum $\rho^*$ of $(\textrm{Q})$ \eqref{sec2-eq9}:
\begin{equation}\label{mixhierc}
\begin{matrix}
\rho^{(1)}_{d_{\min}}&\le&\rho^{(2)}_{d_{\min}}&\le&\cdots&\le&\rho_{d_{\min}}\\
\vge&&\vge&&&&\vge\\
\rho^{(1)}_{d_{\min}+1}&\le&\rho^{(2)}_{d_{\min}+1}&\le&\cdots&\le&\rho_{d_{\min}+1}\\
\vge&&\vge&&&&\vge\\
\vdots&&\vdots&&\vdots&&\vdots\\
\vge&&\vge&&&&\vge\\
\rho^{(1)}_{d}&\le&\rho^{(2)}_{d}&\le&\cdots&\le&\rho_{d}\\
\vge&&\vge&&&&\vge\\
\vdots&&\vdots&&\vdots&&\vdots\\
\end{matrix}
\end{equation}
The array of lower bounds \eqref{mixhierc} (and its associated SDP relaxations \eqref{sec4-eq2}) is what we call the {\em CS-TSSOS} hierarchy associated with $(\textrm{Q})$ \eqref{sec2-eq9}.

\begin{example}\label{ex2}
Let $f=1+\sum_{i=1}^6x_i^4+x_1x_2x_3+x_3x_4x_5+x_3x_4x_6+x_3x_5x_6+x_4x_5x_6$, and consider the unconstrained POP: $\min\{f(\x): \x\in\R^n\}$.
We have $n=6,m=0,d=d_{\min}=2$. Let us apply the CS-TSSOS hierarchy (using the maximal chordal extension in \eqref{sec4-graph}) to this problem. First, according to the csp graph (see Figure \ref{ex2-3}), we partition the variables into two cliques: $\{x_1,x_2,x_3\}$ and $\{x_3,x_4,x_5,x_6\}$. Figure \ref{ex2-1} and Figure \ref{ex2-2} illustrate the tsp graphs for the first clique and the second clique respectively. For the first clique 
one obtains four blocks of SDP matrices with respective sizes $4,2,2,2$. For the second clique one obtains two blocks of SDP matrices with respective sizes $5,10$. As a result, the original SDP matrix of size $28$ has been reduced to six blocks of maximal size $10$.

If one applies the TSSOS hierarchy (using the maximal chordal extension in \eqref{sec2-graph2}) directly to the problem (i.e., without partitioning variables), then the tsp graph is illustrated in Figure \ref{ex2-4}.
One obtains five SDP blocks with respective sizes $7,2,2,2,10$. Compared to the CS-TSSOS case, the two blocks with respective sizes $4,5$ are replaced by a single block of size $7$.

\begin{figure}[htbp]
\caption{The csp graph of Example \ref{ex2}}\label{ex2-3}
\begin{center}
{\tiny
\begin{tikzpicture}[every node/.style={circle, draw=blue!50, thick, minimum size=8mm}]
\node (n1) at (-1.732,1) {$1$};
\node (n2) at (-1.732,-1) {$2$};
\node (n3) at (0,0) {$3$};
\node (n4) at (1.414,1.414) {$4$};
\node (n5) at (1.4147,-1.414) {$5$};
\node (n6) at (2.828,0) {$6$};
\draw (n1)--(n2);
\draw (n1)--(n3);
\draw (n2)--(n3);
\draw (n3)--(n4);
\draw (n3)--(n5);
\draw (n3)--(n6);
\draw (n4)--(n5);
\draw (n4)--(n6);
\draw (n5)--(n6);
\end{tikzpicture}}
\end{center}
\end{figure}
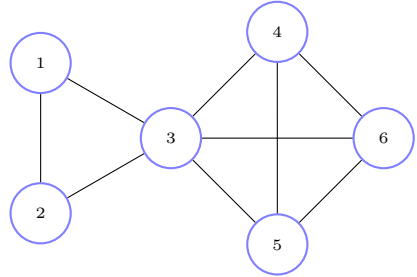

\begin{figure}[htbp]
\caption{The tsp graph for the first clique of Example \ref{ex2}}\label{ex2-1}
\begin{center}
{\tiny
\begin{tikzpicture}[every node/.style={circle, draw=blue!50, thick, minimum size=8mm}]
\node (n1) at (0,0) {$1$};
\node (n8) at (-2,0) {$x_3^2$};
\node (n9) at (-2,-2) {$x_2^2$};
\node (n10) at (0,-2) {$x_1^2$};
\node (n2) at (2,0) {$x_1$};
\node (n3) at (4,0) {$x_2$};
\node (n4) at (6,0) {$x_3$};
\node (n5) at (2,-2) {$x_2x_3$};
\node (n6) at (4,-2) {$x_1x_3$};
\node (n7) at (6,-2) {$x_1x_2$};
\draw (n1)--(n8);
\draw (n1)--(n9);
\draw (n1)--(n10);
\draw (n8)--(n9);
\draw (n8)--(n10);
\draw (n9)--(n10);
\draw (n2)--(n5);
\draw (n3)--(n6);
\draw (n4)--(n7);
\end{tikzpicture}}
\end{center}
\end{figure}
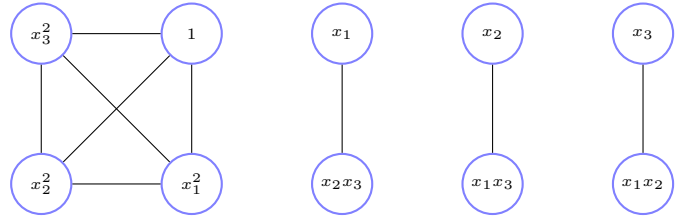

\begin{figure}[htbp]
\caption{The tsp graph for the second clique of Example \ref{ex2}}\label{ex2-2}
\begin{center}
{\tiny
\begin{tikzpicture}[every node/.style={circle, draw=blue!50, thick, minimum size=8mm}]
\node (n1) at (90:2) {$1$};
\node (n2) at (162:2) {$x_6^2$};
\node (n3) at (234:2) {$x_5^2$};
\node (n4) at (306:2) {$x_4^2$};
\node (n5) at (18:2) {$x_3^2$};
\draw (n2)--(n3);
\draw (n2)--(n4);
\draw (n2)--(n5);
\draw (n3)--(n4);
\draw (n3)--(n5);
\draw (n3)--(n1);
\draw (n4)--(n5);
\draw (n4)--(n1);
\draw (n5)--(n1);
\draw (n1)--(n2);
\node[xshift=150] (n6) at (90:2) {$x_3$};
\node[xshift=150] (n7) at (126:2) {$x_5x_6$};
\node[xshift=150] (n8) at (162:2) {$x_4x_6$};
\node[xshift=150] (n9) at (198:2) {$x_4x_5$};
\node[xshift=150] (n10) at (234:2) {$x_3x_6$};
\node[xshift=150] (n12) at (270:2) {$x_3x_5$};
\node[xshift=150] (n13) at (306:2) {$x_3x_4$};
\node[xshift=150] (n14) at (342:2) {$x_6$};
\node[xshift=150] (n15) at (54:2) {$x_4$};
\node[xshift=150] (n16) at (18:2) {$x_5$};
\draw (n6)--(n9);
\draw (n15)--(n12);
\draw (n16)--(n13);
\draw (n6)--(n8);
\draw (n15)--(n10);
\draw (n14)--(n13);
\draw (n6)--(n7);
\draw (n14)--(n12);
\draw (n16)--(n10);
\draw (n15)--(n7);
\draw (n16)--(n8);
\draw (n14)--(n9);
\end{tikzpicture}}
\end{center}
\end{figure}
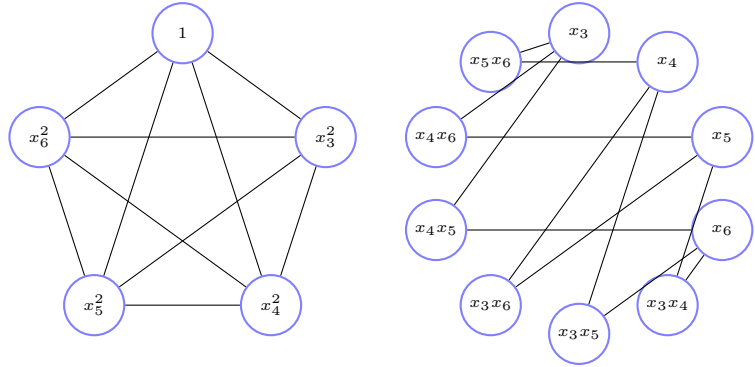 

\begin{figure}[htbp]
\caption{The tsp graph without partitioning variables of Example \ref{ex2}}\label{ex2-4}
\begin{center}
\begin{minipage}{0.4\linewidth}
{\tiny
\begin{tikzpicture}[every node/.style={circle, draw=blue!50, thick, minimum size=8mm}]
\node (n1) at (90:2) {$1$};
\node (n2) at (38.6:2) {$x_1^2$};
\node (n3) at (347:2) {$x_2^2$};
\node (n4) at (295.6:2) {$x_3^2$};
\node (n5) at (244.2:2) {$x_4^2$};
\node (n6) at (192.8:2) {$x_5^2$};
\node (n7) at (141.4:2) {$x_6^2$};
\draw (n2)--(n3);
\draw (n2)--(n4);
\draw (n2)--(n5);
\draw (n3)--(n4);
\draw (n3)--(n5);
\draw (n3)--(n1);
\draw (n4)--(n5);
\draw (n4)--(n5);
\draw (n4)--(n1);
\draw (n5)--(n1);
\draw (n1)--(n2);
\draw (n1)--(n6);
\draw (n2)--(n6);
\draw (n3)--(n6);
\draw (n4)--(n6);
\draw (n5)--(n6);
\draw (n1)--(n7);
\draw (n2)--(n7);
\draw (n3)--(n7);
\draw (n4)--(n7);
\draw (n5)--(n7);
\draw (n6)--(n7);
\end{tikzpicture}}
\end{minipage}%
\begin{minipage}{0.25\linewidth}
{\tiny
\begin{tikzpicture}[every node/.style={circle, draw=blue!50, thick, minimum size=8mm}]
\node (n2) at (1,0) {$x_1$};
\node (n3) at (2,0) {$x_2$};
\node (n4) at (3,0) {$x_3$};
\node (n5) at (1,-2) {$x_2x_3$};
\node (n6) at (2,-2) {$x_1x_3$};
\node (n7) at (3,-2) {$x_1x_2$};
\draw (n2)--(n5);
\draw (n3)--(n6);
\draw (n4)--(n7);
\end{tikzpicture}}
\end{minipage}%
\begin{minipage}{0.42\linewidth}
{\tiny
\begin{tikzpicture}[every node/.style={circle, draw=blue!50, thick, minimum size=8mm}]
\node (n6) at (90:2) {$x_3$};
\node (n7) at (126:2) {$x_5x_6$};
\node (n8) at (162:2) {$x_4x_6$};
\node (n9) at (198:2) {$x_4x_5$};
\node (n10) at (234:2) {$x_3x_6$};
\node (n12) at (270:2) {$x_3x_5$};
\node (n13) at (306:2) {$x_3x_4$};
\node (n14) at (342:2) {$x_6$};
\node (n15) at (54:2) {$x_4$};
\node (n16) at (18:2) {$x_5$};
\draw (n6)--(n9);
\draw (n15)--(n12);
\draw (n16)--(n13);
\draw (n6)--(n8);
\draw (n15)--(n10);
\draw (n14)--(n13);
\draw (n6)--(n7);
\draw (n14)--(n12);
\draw (n16)--(n10);
\draw (n15)--(n7);
\draw (n16)--(n8);
\draw (n14)--(n9);
\end{tikzpicture}}
\end{minipage}
\end{center}
\end{figure}
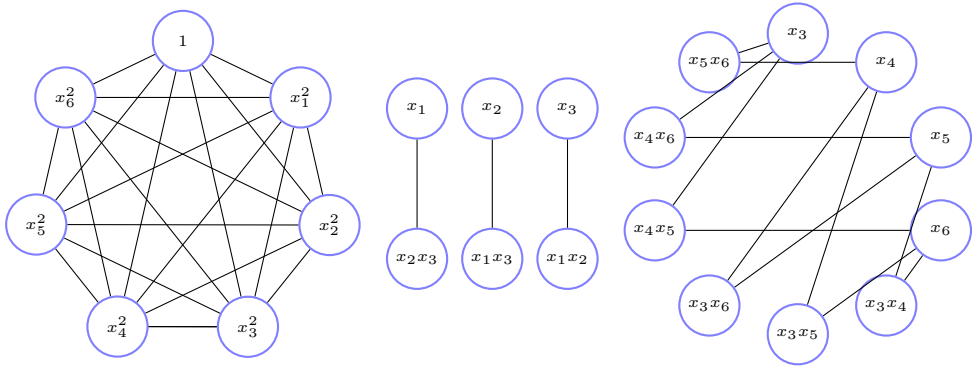
\end{example}

The CS-TSSOS hierarchy entails a trade-off. Indeed, one has the freedom to choose a specific chordal extension for any graph involved in \eqref{sec4-eq2}. This choice affects the resulting size of blocks of SDP matrices and the quality of optimal values of corresponding relaxations. Intuitively, chordal extensions with small clique numbers lead to blocks of small size and optimal values of (possibly) low quality while chordal extensions with large clique numbers lead to blocks of large size and optimal values of (possibly) high quality.

\medskip

For all $l,j$, write $M_{d-d_j}(g_j\y, I_l)=\sum_{\a}D_{\a}^{l,j}y_{\a}$ for appropriate symmetry matrices $\{D_{\a}^{l,j}\}$. Then for each $k\ge1$, the dual of $(\textrm{Q}^{\textrm{cs-ts}}_{d,k})$ reads as:
\begin{equation}\label{sec4-eq3}
(\textrm{Q}^{\textrm{cs-ts}}_{d,k})^*:\quad
\begin{cases}
\sup\,&\rho\\
\textrm{s.t.}\, &\sum_{l=1}^p\sum_{j\in \{0\}\cup J_l}\langle Q_{l,j},D_{\a}^{l,j}\rangle+\rho\delta_{\mathbf{0}\a}=f_{\a},\quad\forall\a\in\C_{d}^{(k)},\\
&Q_{l,j}\in\mathbf{S}_+^{r_{l,j}}\cap\mathbf{S}_{G_{d,l,j}^{(k)}},\quad j\in \{0\}\cup J_l, l=1,\ldots,p,
\end{cases}
\end{equation}
where $\C_{d}^{(k)}$ is defined in \eqref{sec4-eq11}.

\begin{proposition}
Let $f\in\R[\x]$ and $\mathbf{K}$ be as in \eqref{sec2-eq10}. Assume that $\mathbf{K}$ has a nonempty interior. Then there is no duality gap between $(\textrm{Q}^{\textrm{cs-ts}}_{d,k})$ and $(\textrm{Q}^{\textrm{cs-ts}}_{d,k})^*$ for any $d\ge d_{\min}$ and $k\ge1$.
\end{proposition}
\begin{proof}
By the duality theory of convex programming, this easily follows from Theorem 3.6 of \cite{Las06} and Theorem \ref{sec2-thm2}.
\end{proof}

Note that the number of equality constraints in \eqref{sec4-eq3} is equal to the cardinality of $\C_{d}^{(k)}$. We next give a description of the elements in $\C_{d}^{(k)}$ in terms of sign symmetries.

\subsection{Sign symmetries}
\begin{definition}
Given a finite set $\A\subseteq\N^n$, the {\em sign symmetries} of $\A$ are defined by all vectors $\br\in\Z_2^n:=\{0,1\}^n$ such that $\br^T\a\equiv0$ $(\textrm{mod }2)$ for all $\a\in\A$.
\end{definition}
For any $\a\in\N^n$, we define $(\a)_2:=(\alpha_1 (\textrm{mod } 2), \dots, \alpha_n (\textrm{mod } 2))\in\Z_2^n$.
We also use the same notation for any subset $\A\subseteq \N^n$, i.e., $(\A)_2:=\{(\a)_2\mid\a\in\mathscr{A}\}\subseteq\Z_2^n$. For a subset $S\subseteq\Z_2^n$, the {\em orthogonal complement space} of $S$ in $\Z_2^n$, denoted by $S^{\perp}$, is the set $\{\a\in\Z_2^n\mid\a^T\bs\equiv0\,(\textrm{mod }2)\,,\forall\bs\in S\}$.
\begin{remark}
By definition, the set of sign symmetries of $\A$ is exactly the orthogonal complement space $(\A)_2^{\perp}$ in $\Z_2^n$, which therefore can be essentially represented by a basis of the subspace $(\A)_2^{\perp}$ in $\Z_2^n$.
\end{remark}

For a subset $S\subseteq\Z_2^n$, we say that $S$ is {\em closed} under addition modulo $2$ if $\bs_1,\bs_2\in S$ implies $(\bs_1+\bs_2)_2\in S$. The minimal set containing $S$ with elements which are closed under addition modulo $2$ is denoted by $\langle S\rangle_{\Z_2}$. It is easy to prove $\langle S\rangle_{\Z_2}=\{(\sum_i\bs_i)_2\mid\bs_i\in S\}$ which is the subspace spanned by $S$ in $\Z_2^n$.

\begin{lemma}\label{sec4-lm}
Let $S\subseteq\Z_2^n$. Then $(S^{\perp})^{\perp}=\langle S\rangle_{\Z_2}$.
\end{lemma}
\begin{proof}
It is immediate from the definitions.
\end{proof}

\begin{lemma}\label{sec4-lm2}
Suppose $G$ is a graph with $V(G)\subseteq\N^n$. Then it holds $(\supp(G'))_2\subseteq\langle(\supp(G))_2\rangle_{\Z_2}$.
\end{lemma}
\begin{proof}
By definition, we need to show $(\b+\g)_2\in\langle(\supp(G))_2\rangle_{\Z_2}$ for any $\{\b,\g\}\in E(G')$. Since in the process of chordal extensions, edges are added only if two nodes belong to the same connected component, for any $\{\b,\g\}\in E(G')$ there is a path connecting $\b$ and $\g$ in $G$: $\{\b,\up_1,\ldots,\up_r,\g\}$ with $\{\b,\up_1\},\{\up_r,\g\}\in E(G)$ and $\{\up_i,\up_{i+1}\}\in E(G),i=1,\ldots,r-1$. From $(\b+\up_1)_2,(\up_1+\up_2)_2\in (\supp(G))_2$, we deduce that $(\b+\up_2)_2\in\langle(\supp(G))_2\rangle_{\Z_2}$ because $\langle(\supp(G))_2\rangle_{\Z_2}$ is closed under addition modulo $2$. Likewise, we can prove $(\b+\up_i)_2\in\langle(\supp(G))_2\rangle_{\Z_2}$ for $i=3,\ldots,r+1$ with $\up_{r+1}:=\g$. Hence $(\b+\g)_2\in\langle(\supp(G))_2\rangle_{\Z_2}$ as desired.
\end{proof}

\begin{proposition}\label{prop-ss}
Let $\A$ be defined as in \eqref{sec4-eq0}, $\C_{d}^{(k)}$ be defined as in \eqref{sec4-eq11} and assume that the sign symmetries of $\mathscr{A}$ are represented by the column vectors of a binary matrix, denoted by $R$. Then for any $k\ge1$ and any $\a\in\C_{d}^{(k)}$, it holds $R^T\a\equiv0\,(\textrm{mod }2)\,$. In other words, $(\C_{d}^{(k)})_2\subseteq R^{\perp}$, where we regard $R$ as a set of its column vectors.
\end{proposition}
\begin{proof}
By Lemma \ref{sec4-lm}, we only need to prove $(\C_{d}^{(k)})_2\subseteq\langle(\A)_2\rangle_{\Z_2}$. Let us do induction on $k\ge0$. For $k=0$, by \eqref{sec4-eq11}, $\C_{d}^{(0)}=\bigcup_{l=1}^p\supp(G_{d,l,0}^{(0)})=\bigcup_{l=1}^p\supp(G_{d,l}^{\textrm{tsp}})\subseteq\bigcup_{l=1}^p(\A_l\cup(2\N)^{n_l})\subseteq\A\cup(2\N)^n$. Hence $(\C_{d}^{(0)})_2\subseteq\langle(\A)_2\rangle_{\Z_2}$.
Now assume that $(\C_{d}^{(k)})_2\subseteq\langle(\A)_2\rangle_{\Z_2}$ holds for some $k\ge0$. By \eqref{sec4-eq1}, for any $l,j$ and any $\{\b,\g\}\in E(F_{d,l,j}^{(k+1)})$, we have $(\supp(g_j)+\b+\g)\cap\C_{d}^{(k)}\ne\emptyset$, i.e., there exists $\a\in\supp(g_j)$ such that $\a+\b+\g\in\C_{d}^{(k)}$, which implies $(\a+\b+\g)_2\in(\C_{d}^{(k)})_2$. Hence by the induction hypothesis, $(\a+\b+\g)_2\in\langle(\A)_2\rangle_{\Z_2}$. Since $\langle(\A)_2\rangle_{\Z_2}$ is closed under addition modulo $2$ and $(\a)_2\in (\A)_2$, we have $(\b+\g)_2\in\langle(\A)_2\rangle_{\Z_2}$. It follows $(\supp(F_{d,l,j}^{(k+1)}))_2\subseteq\langle(\A)_2\rangle_{\Z_2}$. Because $G_{d,l,j}^{(k+1)}=(F_{d,l,j}^{(k+1)})'$, by Lemma \ref{sec4-lm2}, we have $(\supp(G_{d,l,j}^{(k+1)}))_2\subseteq\langle(\supp(F_{d,l,j}^{(k+1)}))_2\rangle_{\Z_2}\subseteq\langle(\A)_2\rangle_{\Z_2}$. From this, \eqref{sec4-eq11} and the fact that $\langle(\A)_2\rangle_{\Z_2}$ is closed under addition modulo $2$, we then deduce the inclusion $(\C_{d}^{(k+1)})_2\subseteq\langle(\A)_2\rangle_{\Z_2}$ which completes the induction.
\end{proof}

\revision{\begin{remark}
Proposition \ref{prop-ss} actually indicates that the block structure produced by the CS-TSSOS hierarchy is consistent with the sign symmetries of the POP.
\end{remark}}


\section{Convergence analysis}
\label{sec:cvg}
\subsection{Global convergence}\label{glo-con}
We next prove that if for any graph involved in \eqref{sec4-eq2}, the chordal extension is chosen to be \emph{maximal}, then 
for any relaxation order $d\ge d_{\min}$ the sequence of optimal values $(\rho^{(k)}_{d})_{k\ge1}$ of the CS-TSSOS hierarchy 
converges to the optimal value $\rho_{d}$ of the corresponding CSSOS hierarchy \eqref{sec2-eq13}. In turn, as the relaxation order $d$ 
increases, the latter sequence 
converges to the global optimum $\rho^*$ of the original POP \eqref{sec2-eq9} (after adding some redundant quadratic constraints) as shown in \cite{Las06}.

Obviously, the sequences of graphs $(G_{d,l,j}^{(k)}(V_{d,l,j},E_{d,l,j}^{(k)}))_{k\ge1}$ stabilize for all $j\in \{0\}\cup J_l,l=1,\ldots,p$ after finitely many steps. We denote the resulting stabilized graphs by $G_{d,l,j}^{(*)},j\in \{0\}\cup J_l,l=1,\ldots,p$ and the corresponding SDP \eqref{sec4-eq2} by $(\textrm{Q}^{\textrm{cs-ts}}_{d,*})$.
\begin{theorem}\label{sec4-thm1}
Assume that the chordal extension in \eqref{sec4-graph} is the maximal chordal extension. Then for any $d\ge d_{\min}$, the sequence $(\rho^{(k)}_{d})_{k\ge1}$ converges to $\rho_{d}$ in finitely many  steps.
\end{theorem}
\begin{proof}
Let $\y=(y_{\a})$ be an arbitrary feasible solution of $(\textrm{Q}^{\textrm{cs-ts}}_{d,*})$ and $\rho_{d}^*$ be the optimal value of $(\textrm{Q}^{\textrm{cs-ts}}_{d,*})$. Note that $\{y_{\a}\mid\a\in\bigcup_{l=1}^p(\cup_{j\in \{0\}\cup J_l}(\supp(g_j)+\supp(G_{d,l,j}^{(*)})))\}$ is the set of decision variables involved in $(\textrm{Q}^{\textrm{cs-ts}}_{d,*})$. Let $\RR$ be the set of decision variables involved in $(\textrm{Q}^{\textrm{cs}}_{d})$ \eqref{sec2-eq13}.
We then define a vector $\overline{\y}=(\overline{y}_{\a})_{\a\in\RR}$ as follows:
$$\overline{y}_{\a}=\begin{cases}y_{\a},\quad\textrm{ if }\a\in\bigcup_{l=1}^p(\cup_{j\in \{0\}\cup J_l}(\supp(g_j)+\supp(G_{d,l,j}^{(*)}))),\\
0,\quad\quad\textrm{otherwise}.
\end{cases}$$
\revision{By construction and since $G_{d,l,j}^{(*)}$ stabilizes under support extension for all $l,j$, we have $M_{d-d_j}(g_j\overline{\y},I_l)=B_{G_{l,j,d}^{(*)}}\circ M_{d-d_j}(g_j\y,I_l)$.
As we use the maximal chordal extension in \eqref{sec4-graph}, the matrix $B_{G_{l,j,d}^{(*)}}\circ M_{d-d_j}(g_j\y,I_l)$ is block diagonal up to permutation (see Remark \ref{rm-max}). So from $B_{G_{l,j,d}^{(*)}}\circ M_{d-d_j}(g_j\y, I_l)\in\Pi_{G_{l,j,d}^{(*)}}(\mathbf{S}_+^{r_{l,j}})$ it follows $M_{d-d_j}(g_j\overline{\y},I_l)\succeq0$ for $j\in \{0\}\cup J_l, l=1,\ldots,p$.} Therefore $\overline{\y}$ is a feasible solution of $(\textrm{Q}^{\textrm{cs}}_{d})$ and so $L_{\y}(f)=L_{\overline{\y}}(f)\ge\rho_{d}$.
Hence $\rho^{*}_{d}\ge\rho_{d}$ since $\y$ is an arbitrary feasible solution of $(\textrm{Q}^{\textrm{cs-ts}}_{d,*})$. 
By Proposition \ref{sec4-prop1}, we already have $\rho^{*}_{d}\le\rho_{d}$. Therefore, $\rho^{*}_{d}=\rho_{d}$.
\end{proof}

To guarantee the global optimality, we need the following compactness assumption on the feasible set $\mathbf{K}$.

\medskip
\noindent{\bf Assumption 1.} Let $\mathbf{K}$ be as in \eqref{sec2-eq10}. There exists an $M>0$ such that $||\x||_{\infty} <M$ for all $\x\in\mathbf{K}$.
\medskip

Because of Assumption 1, one has $||\x(I_l)||_2^2\le n_lM^2$, $l=1,\ldots,p$. Therefore, we can add the $p$ redundant quadratic constraints
\begin{equation}\label{sec4-eq5}
    g_{m+l}(\x):=n_lM^2-||\x(I_l)||_2^2\ge0,\quad l=1,\ldots,p
\end{equation}
in the definition \eqref{sec2-eq10} of $\mathbf{K}$ and set $m' = m + p$, so that $\mathbf{K}$ is now defined by
\begin{equation}\label{sec4-eq6}
\mathbf{K}:=\{\x\in\R^n\mid g_j(\x)\ge0,\quad j=1,\ldots,m'\}.
\end{equation}
Note that $g_{m+l}\in\R[\x(I_l)]$ for $l=1,\ldots,p$.

Then by Theorem 3.6 in \cite{Las06}, the sequence $(\rho_{d})_{d\ge d_{\min}}$ converges to the globally optimal value $\rho^*$ of (Q) \eqref{sec2-eq9}. So this together with Theorem \ref{sec4-thm1} gives the global convergence of the CS-TSSOS hierarchy.

\subsection{A sparse representation theorem}
Proceeding along Theorem \ref{sec4-thm1}, we 
are able to provide a \emph{sparse representation} theorem for a polynomial positive on a compact basic semialgebraic set. 

\begin{theorem}[sparse representation]\label{sec4-thm2}
Let $f\in\R[\x]$ and $\mathbf{K}$ be as in \eqref{sec4-eq6} with the additional quadratic constraints \eqref{sec4-eq5}. Let $I_l,J_l$ be defined as in Section~\ref{sec4-1} and $\mathscr{A}=\supp(f)\cup\bigcup_{j=1}^{m'}\supp(g_j)$. Assume that the sign symmetries of $\mathscr{A}$ are represented by the column vectors of the binary matrix $R$. If $f$ is positive on $\mathbf{K}$, then
\begin{equation}
\label{sparse-certificate}
f=\sum_{l=1}^p\left(\sigma_{l,0}+\sum_{j\in J_l}\sigma_{l,j}g_j\right),
\end{equation}
for some polynomials $\sigma_{l,j}\in\Sigma[\x(I_l)],j\in \{0\}\cup J_l, l=1,\ldots,p$, satisfying $R^T\a\equiv0$ $(\textrm{mod }2)$ for any $\a\in\supp(\sigma_{l,j})$, i.e., $(\supp(\sigma_{l,j}))_2\subseteq R^{\perp}$, where we regard $R$ as a set of its column vectors.

That is, \eqref{sparse-certificate} provides a certificate of positivity of $f$  on $\mathbf{K}$.
\end{theorem}
\begin{proof}
By Corollary 3.9 of \cite{Las06} (or Theorem 5 of \cite{grim}), there exist polynomials $\sigma'_{l,j}\in\Sigma[\x(I_l)],j\in \{0\}\cup J_l, l=1,\ldots,p$ such that
\begin{equation}\label{sec4-eq7}
f=\sum_{l=1}^p\left(\sigma'_{l,0}+\sum_{j\in J_l}\sigma'_{l,j}g_j\right).
\end{equation}
Let $d=\max\{\lceil\deg(\sigma'_{l,j}g_j)/2\rceil: j\in \{0\}\cup J_l, l=1,\ldots,p\}$. Let $Q'_{l,j}$ be a PSD Gram matrix associated with $\sigma'_{l,j}$ and indexed by the monomial basis $\N^{n_l}_{d-d_j}$. Then for all $l,j$, we define $Q_{l,j}\in\mathbf{S}^{r_{l,j}}$ with $r_{l,j}=\binom{n_l+d-d_j}{d-d_j}$ (indexed by $\N^{n_l}_{d-d_j}$) by
\begin{equation*}
    [Q_{l,j}]_{\b\g}:=\begin{cases}
    [Q'_{l,j}]_{\b\g}, \quad\textrm{if }R^T(\b+\g)\equiv0$ $(\textrm{mod }2),\\
    0, \quad\quad\quad\quad\textrm{otherwise,}
    \end{cases}
\end{equation*}
and let $\sigma_{l,j}=(\x^{\N^{n_l}_{d-d_j}})^TQ_{l,j}\x^{\N^{n_l}_{d-d_j}}$.
One can easily verify that $Q_{l,j}$ is block diagonal up to permutation (see also \cite{wang2}) and each block is a principal submatrix of $Q'_{l,j}$. Then the positive semidefiniteness of $Q'_{l,j}$ implies that $Q_{l,j}$ is also positive semidefinite. Thus $\sigma_{l,j}\in\Sigma[\x(I_l)]$.

By construction, substituting $\sigma'_{l,j}$ with $\sigma_{l,j}$ in \eqref{sec4-eq7} boils down to removing the terms with exponents $\a$ that do not satisfy $R^T\a\equiv0$ $(\textrm{mod }2)$ from the right hand side of \eqref{sec4-eq7}. Since any $\a\in\supp(f)$ satisfies $R^T\a\equiv0$ $(\textrm{mod }2)$, this does not change the match of coefficients on both sides of the equality. Thus we obtain
\begin{equation*}
f=\sum_{l=1}^p\left(\sigma_{l,0}+\sum_{j\in J_l}\sigma_{l,j}g_j\right)
\end{equation*}
with the desired property.
\end{proof}

\subsection{Extracting a solution} \label{sec:extract}
In the case of dense moment-SOS relaxations, there is a standard procedure described in \cite{henrion2005detecting} to extract globally optimal solutions when the so-called flatness condition of the moment matrix is satisfied, and this procedure is also generalized to the correlative sparsity setting in \cite[\textsection~3.3]{Las06}. However, in the combined sparsity setting, the corresponding procedure cannot be directly applied because we do not have full information on the moment matrix associated with each clique. 
In order to extract a solution in this case, we may add a dense moment matrix of order one for each clique in \eqref{sec4-eq2}:
\begin{equation}\label{sec4-eq8}
(\textrm{Q}^{\textrm{cs-ts}}_{d,k})':\quad
\begin{cases}
\inf&L_{\y}(f)\\
\textrm{s.t.}&B_{G_{d,l,0}^{(k)}}\circ M_{d}(\y, I_l)\in\Pi_{G_{d,l,0}^{(k)}}(\mathbf{S}_+^{r_{l,0}}),\quad l=1,\ldots,p,\\
&M_{1}(\y, I_l)\succeq0,\quad l=1,\ldots,p,\\
&B_{G_{d,l,j}^{(k)}}\circ M_{d-d_j}(g_j\y, I_l)\in\Pi_{G_{d,l,j}^{(k)}}(\mathbf{S}_+^{r_{l,j}}),\quad j\in J_l,l=1,\ldots,p,\\
&y_{\mathbf{0}}=1.
\end{cases}
\end{equation}

Let $\y^*$ be an optimal solution of $(\textrm{Q}^{\textrm{cs-ts}}_{d,k})'$. Typically, $M_{1}(\y^*, I_l)$ (after identifying sufficiently small entries with zero) is a block diagonal matrix (up to permutation). If for all $l$, every block of $M_{1}(\y^*, I_l))$ has rank one, then a globally optimal solution $\x^*$ to (Q) \eqref{sec2-eq9} can be extracted and the global optimality is certified (see \cite[Theorem~3.2]{Las06}). Otherwise, the relaxation might be not exact or yield multiple global solutions. In the latter case, adding a small perturbation to the objective function, as in \cite{waki}, may yield a unique global solution.

\begin{remark}
Note that $(\textrm{Q}^{\textrm{cs-ts}}_{d,k})'$ is a tighter 
relaxation of $(\textrm{Q})$ than $(\textrm{Q}^{\textrm{cs-ts}}_{d,k})$ and so might provide a better lower bound for $(\textrm{Q})$.
\end{remark}

\section{Applications and numerical experiments}
\label{sec:benchs}
In this section, we conduct numerical experiments for the proposed CS-TSSOS hierarchy and apply it to two important classes of POPs: Max-Cut problems and AC optimal power flow (AC-OPF) problems. Depending on specific problems, we consider two types of chordal extensions for the term sparsity pattern: maximal chordal extensions and approximately smallest chordal extensions\footnote{A smallest chordal extension is a chordal extension with the smallest clique number. Computing a smallest chordal extension is generally NP-complete. So in practice we compute approximately smallest chordal extensions instead with efficient heuristic algorithms.}. The tool {\tt TSSOS} which executes the CS-TSSOS hierarchy (as well as the CSSOS hierarchy and the TSSOS hierarchy) is implemented in Julia. For an introduction to {\tt TSSOS}, one could refer to \cite{magron2021tssos}.
{\tt TSSOS} is available on the website:

\begin{center}
\url{https://github.com/wangjie212/TSSOS}.
\end{center}

In the following subsections, we compare the performances of the CSSOS approach, the TSSOS approach, the CS-TSSOS approach and the SDSOS approach \cite{ahmadi2019dsos} (implemented in {\tt SPOT} \cite{spot}).
{\tt Mosek} \cite{mosek} is used as an SDP (in the CSSOS, TSSOS, CS-TSSOS cases) or SOCP (in the SDSOS case) solver. All numerical examples were computed on an Intel Core i5-8265U@1.60GHz CPU with 8GB RAM memory. The timing includes the time required to generate the SDP/SOCP and the time spent to solve it. The notations used in this section are listed in Table \ref{not}.

\begin{table}[htbp]
\caption{Notation}\label{not}
\begin{center}
\begin{tabular}{|c|c|}
\hline
var&number of variables\\
\hline
cons&number of constraints\\
\hline
mc&maximal size of variable cliques\\
\hline
mb&maximal size of SDP blocks\\
\hline
opt&optimal value\\
\hline
time&running time in seconds\\
\hline
gap&optimality gap\\
\hline
CE&type of chordal extensions used in \eqref{sec4-graph}\\
\hline
min&approximately smallest chordal extension\\
\hline
max&maximal chordal extension\\
\hline
0&a number whose absolute value less than $1\text{e-}5$\\
\hline
-&an out of memory error\\
\hline
\end{tabular}
\end{center}
\end{table}

\subsection{Benchmarks for unconstrained POPs}
The Broyden banded function is defined as
\begin{equation*}
    f_{\textrm{Bb}}(\x)=\sum_{i=1}^n(x_i(2+5x_i^2)+1-\sum_{j\in J_i}(1+x_j)x_j)^2,
\end{equation*}
where $J_i=\{j\mid j\ne i, \max(1,i-5)\le j\le\min(n,i+1)\}$.

The task is to minimize the Broyden banded function over $\R^n$ which is formulated as an unconstrained POP. Using the relaxation order $d=3$, we solve the CSSOS hierarchy $(\textrm{Q}^{\textrm{cs}}_{d})$ \eqref{sec2-eq13}, the TSSOS hierarchy $(\textrm{Q}^{\textrm{ts}}_{d,k})$ \eqref{sec2-eq16} with $k=1$ and the CS-TSSOS hierarchy $(\textrm{Q}^{\textrm{cs-ts}}_{d,k})$ \eqref{sec4-eq2} with $k=1$. In the latter two cases, approximately smallest chordal extensions are used. We also solve the POP with the SDSOS approach. The results are displayed in Table \ref{bb}. 

\revision{It can be seen from the table that CS-TSSOS significantly reduces the maximal size of SDP blocks and is the most efficient approach. CSSOS, TSSOS and CS-TSSOS all give the exact minimum $0$ while SDSOS only gives a very loose lower bound $-13731$ when $n=20$. Due to the limitation of memory, CSSOS scales up to $180$ varables; TSSOS scales up to $40$ varables; SDSOS scales up to $20$ varables. On the other hand, CS-TSSOS can easily handle instances with up to $500$ variables.}

\begin{table}[htbp]
\caption{The result for Broyden banded functions ($d=3$)}\label{bb}
\begin{center}
\begin{tabular}{|c|c|c|c|c|c|c|c|c|c|c|c|}
\hline
\multirow{2}*{var}&\multicolumn{3}{c|}{CSSOS}&\multicolumn{3}{c|}{TSSOS}&\multicolumn{3}{c|}{CS-TSSOS}&\multicolumn{2}{c|}{SDSOS}\\
\cline{2-12}
&mb&opt&time&mb&opt&time&mb&opt&time&opt&time\\
\hline
20&$120$&0&$21.7$&$33$&0&$4.39$&$19$&0&$2.24$&$-13731$&$374$\\
\hline
40&$120$&0&$44.6$&$52$&0&$231$&$19$&0&$6.95$&-&-\\
\hline
60&$120$&0&$81.8$&-&-&-&$19$&0&$13.0$&-&-\\
\hline
80&$120$&0&$116$&-&-&-&$19$&0&$19.6$&-&-\\
\hline
100&$120$&0&$151$&-&-&-&$19$&0&$27.0$&-&-\\
\hline
120&$120$&0&195&-&-&-&$19$&0&$34.4$&-&-\\
\hline
140&$120$&0&249&-&-&-&$19$&0&$43.1$&-&-\\
\hline
160&$120$&0&298&-&-&-&$19$&0&$50.2$&-&-\\
\hline
180&$120$&0&338&-&-&-&$19$&0&$63.8$&-&-\\
\hline
200&$120$&-&-&-&-&-&$19$&0&$72.9$&-&-\\
\hline
250&$120$&-&-&-&-&-&$19$&0&$106$&-&-\\
\hline
300&$120$&-&-&-&-&-&$19$&0&$132$&-&-\\
\hline
400&$120$&-&-&-&-&-&$19$&0&$220$&-&-\\
\hline
500&$120$&-&-&-&-&-&$19$&0&$313$&-&-\\
\hline
\end{tabular}
\end{center}
\end{table}

\subsection{Benchmarks for constrained POPs}
~\\

$\bullet$ The generalized Rosenbrock function
\begin{equation*}
    f_{\textrm{gR}}(\x)=1+\sum_{i=2}^n(100(x_i-x_{i-1}^2)^2+(1-x_i)^2).
\end{equation*}

$\bullet$ The Broyden tridiagonal function
\begin{align*}
    f_{\textrm{Bt}}(\x)=&((3-2x_1)x_1-2x_2+1)^2+\sum_{i=2}^{n-1}((3-2x_i)x_i-x_{i-1}-2x_{i+1}+1)^2\\&+((3-2x_n)x_n-x_{n-1}+1)^2.
\end{align*}

$\bullet$ The chained Wood function
\begin{align*}
    f_{\textrm{cW}}(\x)=&1+\sum_{i\in J}(100(x_{i+1}-x_{i}^2)^2+(1-x_i)^2+90(x_{i+3}-x_{i+2}^2)^2\\
    &+(1-x_{i+2})^2+10(x_{i+1}+x_{i+3}-2)^2+0.1(x_{i+1}-x_{i+3})^2),
\end{align*}
where $J=\{1,3,5,\ldots,n-3\}$ and $4|n$.

With the generalized Rosenbrock (resp. Broyden tridiagonal or chained Wood) function as the objective function, we consider the following constrained POP:
\begin{equation}\label{cons}
    \begin{cases}
    \inf&f_{\textrm{gR}}\quad(\textrm{resp. } f_{\textrm{Bt}}\textrm{ or } f_{\textrm{cW}})\\
    \textrm{s.t.}&1-(\sum_{i=20j-19}^{20j}x_i^2)\ge0,\quad j=1,2,\ldots,n/20,
    \end{cases}
\end{equation}
where $20|n$. The generalized Rosenbrock function, the Broyden tridiagonal function and the chained Wood function involve cliques with 2 or 3 variables, which can be efficiently handled by the CSSOS hierarchy; see \cite{waki}. For them, the CS-TSSOS hierarchy gives almost the same results with the CSSOS hierarchy. Hence we add the sphere constraints in \eqref{cons} to increase the clique size and to show the difference.

For these problems, the minimum relaxation order $d=2$ is used. As in the unconstrained case, we solve the CSSOS hierarchy $(\textrm{Q}^{\textrm{cs}}_{d})$ \eqref{sec2-eq13}, the TSSOS hierarchy $(\textrm{Q}^{\textrm{ts}}_{d,k})$ \eqref{sec2-eq16} with $k=1$ and the CS-TSSOS hierarchy $(\textrm{Q}^{\textrm{cs-ts}}_{d,k})$ \eqref{sec4-eq2} with $k=1$, and use approximately smallest chordal extensions. 
We also solve these POPs with the SDSOS approach. The results are displayed in Table \ref{gr}--\ref{cw}. 

\revision{From these tables, one can see that CS-TSSOS significantly reduces the maximal size of SDP blocks and is again the most efficient approach. For the generalized Rosenbrock function, CSSOS, TSSOS and CS-TSSOS give almost the same optimum while SDSOS gives a slightly loose lower bound (only for $n=40$); for the Broyden tridiagonal function, CSSOS, TSSOS and CS-TSSOS all give the same optimum while SDSOS gives a very loose lower bound (only for $n=40$); for the chained Wood function, CSSOS, TSSOS and CS-TSSOS all give the same optimum while SDSOS gives a slightly loose lower bound (only for $n=40$). Due to the limitation of memory, CSSOS scales up to $180$ varables; TSSOS scales up to $180$ or $200$ varables; SDSOS scales up to $40$ varables. On the other hand, CS-TSSOS can easily handle these instances with up to $1000$ variables.}

\begin{table}[htbp]
\caption{The result for the generalized Rosenbrock function ($d=2$)}\label{gr}
\begin{center}{\small
\begin{tabular}{|c|c|c|c|c|c|c|c|c|c|c|c|}
\hline
\multirow{2}*{var}&\multicolumn{3}{c|}{CSSOS}&\multicolumn{3}{c|}{TSSOS}&\multicolumn{3}{c|}{CS-TSSOS}&\multicolumn{2}{c|}{SDSOS}\\
\cline{2-12}
&mb&opt&time&mb&opt&time&mb&opt&time&opt&time\\
\hline
40&$231$&$38.051$&$126$&$41$&38.049&0.61&$21$&$38.049$&$0.23$&$37.625$&$115$\\
\hline
60&$231$&$57.849$&$232$&61&57.845&3.31&$21$&$57.845$&$0.32$&-&-\\
\hline
80&$231$&$77.647$&$306$&81&77.641&11.7&$21$&$77.641$&$0.41$&-&-\\
\hline
100&$231$&97.445&377&101&97.436&31.3&$21$&$97.436$&$0.54$&-&-\\
\hline
120&$231$&117.24&408&121&117.23&75.4&$21$&$117.23$&$0.60$&-&-\\
\hline
140&$231$&137.04&495&141&137.03&190&$21$&$137.03$&$0.75$&-&-\\
\hline
160&$231$&156.84&570&161&156.82&367&$21$&$156.82$&$0.90$&-&-\\
\hline
180&$231$&176.64&730&181&176.62&628&$21$&$176.62$&$1.09$&-&-\\
\hline
200&$231$&-&-&201&196.41&1327&$21$&$196.41$&$1.27$&-&-\\
\hline
300&$231$&-&-&-&-&-&$21$&$295.39$&$2.26$&-&-\\
\hline
400&$231$&-&-&-&-&-&$21$&$394.37$&$3.36$&-&-\\
\hline
500&$231$&-&-&-&-&-&$21$&$493.35$&$4.65$&-&-\\
\hline
1000&$231$&-&-&-&-&-&$21$&$988.24$&$15.8$&-&-\\
\hline
\end{tabular}}
\end{center}
\end{table}

\begin{table}[htbp]
\caption{The result for the Broyden tridiagonal function ($d=2$)}\label{bt}
\begin{center}{\small
\begin{tabular}{|c|c|c|c|c|c|c|c|c|c|c|c|}
\hline
\multirow{2}*{var}&\multicolumn{3}{c|}{CSSOS}&\multicolumn{3}{c|}{TSSOS }&\multicolumn{3}{c|}{CS-TSSOS}&\multicolumn{2}{c|}{SDSOS}\\
\cline{2-12}
&mb&opt&time&mb&opt&time&mb&opt&time&opt&time\\
\hline
40&$231$&$31.234$&$168$&43&31.234&1.95&$23$&$31.234$&$0.64$&$-5.8110$&$138$\\
\hline
60&$231$&$47.434$&$273$&63&47.434&8.33&$23$&$47.434$&$1.14$&-&-\\
\hline
80&$231$&$63.634$&$413$&83&63.634&33.9&$23$&$63.634$&$1.50$&-&-\\
\hline
100&$231$&79.834&519&103&79.834&104&$23$&$79.834$&$1.96$&-&-\\
\hline
120&$231$&96.034&671&123&96.034&199&$23$&$96.034$&$2.30$&-&-\\
\hline
140&$231$&112.23&872&143&112.23&490&$23$&$112.23$&$2.94$&-&-\\
\hline
160&$231$&128.43&1002&163&128.43&783&$23$&$128.43$&$3.67$&-&-\\
\hline
180&$231$&144.63&1066&183&144.63&1329&$23$&$144.63$&$4.46$&-&-\\
\hline
200&$231$&-&-&-&-&-&$23$&$160.83$&$4.88$&-&-\\
\hline
300&$231$&-&-&-&-&-&$23$&$241.83$&$8.67$&-&-\\
\hline
400&$231$&-&-&-&-&-&$23$&$322.83$&$13.3$&-&-\\
\hline
500&$231$&-&-&-&-&-&$23$&$403.83$&$19.9$&-&-\\
\hline
1000&$231$&-&-&-&-&-&$23$&$808.83$&$57.5$&-&-\\
\hline
\end{tabular}}
\end{center}
\end{table}

\begin{table}[htbp]
\caption{The result for the chained Wood function ($d=2$)}\label{cw}
\begin{center}{\small
\begin{tabular}{|c|c|c|c|c|c|c|c|c|c|c|c|}
\hline
\multirow{2}*{var}&\multicolumn{3}{c|}{CSSOS}&\multicolumn{3}{c|}{TSSOS}&\multicolumn{3}{c|}{CS-TSSOS}&\multicolumn{2}{c|}{SDSOS}\\
\cline{2-12}
&mb&opt&time&mb&opt&time&mb&opt&time&opt&time\\
\hline
40&$231$&$574.51$&$164$&41&574.51&0.81&$21$&$574.51$&$0.26$&$518.11$&$110$\\
\hline
60&$231$&$878.26$&$254$&61&878.26&3.61&$21$&$878.26$&$0.40$&-&-\\
\hline
80&$231$&$1182.0$&$393$&81&1182.0&15.3&$21$&$1182.0$&$0.57$&-&-\\
\hline
100&$231$&1485.8&505&101&1485.8&43.2&$21$&$1485.8$&$0.73$&-&-\\
\hline
120&$231$&1789.5&516&121&1789.5&88.4&$21$&$1789.5$&$0.93$&-&-\\
\hline
140&$231$&2093.3&606&141&2093.3&195&$21$&$2093.3$&$1.16$&-&-\\
\hline
160&$231$&2397.0&700&161&2397.0&403&$21$&$2397.0$&$1.39$&-&-\\
\hline
180&$231$&2700.8&797&181&2700.8&867&$21$&$2700.8$&$1.54$&-&-\\
\hline
200&$231$&-&-&201&3004.5&1238&$21$&$3004.5$&$1.91$&-&-\\
\hline
300&$231$&-&-&-&-&-&$21$&$4523.6$&$3.39$&-&-\\
\hline
400&$231$&-&-&-&-&-&$21$&$6042.0$&$5.72$&-&-\\
\hline
500&$231$&-&-&-&-&-&$21$&$7560.7$&$7.88$&-&-\\
\hline
1000&$231$&-&-&-&-&-&$21$&$15155$&$23.0$&-&-\\
\hline
\end{tabular}}
\end{center}
\end{table}

\subsection{The Max-Cut problem}
The Max-Cut problem is one of the basic combinatorial optimization problems, which is known to be NP-hard. Let $G(V, E)$ be an undirected graph with $V=\{1,\ldots,n\}$ and with edge weights $w_{ij}$ for $\{i,j\}\in E$. Then the Max-Cut problem for $G$ can be formulated as a QCQP in binary variables:
\begin{equation}\label{maxcut}
    \begin{cases}
    \inf&\frac{1}{2}\sum_{\{i,j\}\in E}w_{ij}(1-x_ix_j)\\
    \textrm{s.t.}&1-x_i^2=0,\quad i=1,\ldots,n.
    \end{cases}
\end{equation}
\revision{The property of binary variables in \eqref{maxcut} can be also exploited to reduce the size of SDPs arising in the moment-SOS hierarchy, which has been implemented in {\tt TSSOS}.}

For the numerical experiments, we construct random instances of Max-Cut problems with a block-band sparsity pattern (illustrated in Figure \ref{block-band}) which consists of $l$ blocks of size $b$ and two bands of width $h$. 
Here we select $b=25$ and $h=5$. 
For a given $l$, we generate a random sparse binary matrix $A\in\mathbf{S}^{lb+h}$ according to the block-arrow sparsity pattern: the entries out of the blue area take zero; the entries in the block area take one with probability $0.16$; the entries in the band area take one with probability $2/\sqrt{l}$. Then we construct the graph $G$ with $A$ as its adjacency matrix. 
For each edge $\{i,j\}\in E(G)$, the weight $w_{ij}$ randomly takes values $1$ or $-1$ with equal probability. Doing so, we build $10$ Max-Cut instances with $l=20,40,60,80,100,120,140,160,180,200$, respectively\footnote{The instances are available at https://wangjie212.github.io/jiewang/code.html.}. The largest number of nodes is $5005$.

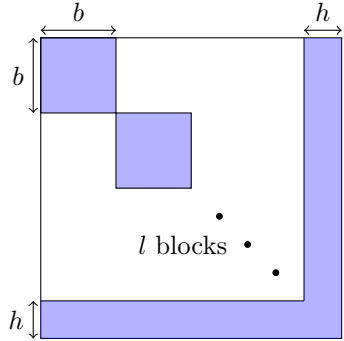
\begin{figure}[htbp]
\caption{The block-band sparsity pattern}\label{block-band}
\begin{center}
\begin{tikzpicture}
\draw (0,0) rectangle (4,4);
\filldraw[fill=blue, fill opacity=0.3] (0,3) rectangle (1,4);
\filldraw[fill=blue, fill opacity=0.3] (1,2) rectangle (2,3);
\fill (2.375,1.625) circle (0.3ex);
\fill (2.75,1.25) circle (0.3ex);
\fill (3.125,0.875) circle (0.3ex);
\draw (0,0.5)--(3.5,0.5);
\draw (3.5,4)--(3.5,0.5);
\fill[fill=blue, fill opacity=0.3] (0,0) rectangle (3.5,0.5);
\fill[fill=blue, fill opacity=0.3] (3.5,0) rectangle (4,4);
\draw[<->] (-0.1,0) -- (-0.1,0.5);
\node[left] at (-0.1,0.25) {$h$};
\draw[<->] (3.5,4.1) -- (4,4.1);
\node[above] at (3.75,4.1) {$h$};
\draw[<->] (-0.1,3) -- (-0.1,4);
\node[left] at (-0.1,3.5) {$b$};
\draw[<->] (0,4.1) -- (1,4.1);
\node[above] at (0.5,4.1) {$b$};
\node[left] at (2.6,1.25) {$l$ blocks};
\end{tikzpicture}\\
{\small $l$: the number of blocks; $b$: the size of blocks; $h$: the width of bands.}
\end{center}
\end{figure}

For each instance, we solve the first-order moment-SOS relaxation (Shor's relaxation), the CSSOS hierarchy with $d=2$, and the CS-TSSOS hierarchy with $d=2,k=1$ for which the maximal chordal extension is used. The results are displayed in Table \ref{max-cut}. From the table we can see that for each instance, both CSSOS and CS-TSSOS significantly improve the bound obtained by Shor's relaxation. Meanwhile, CS-TSSOS is several times faster than CSSOS at the cost of possibly providing a sightly weaker bound.

\begin{table}[htbp]
\caption{The result for Max-Cut instances}\label{max-cut}
\begin{center}
\begin{tabular}{|c|c|c|c|c|c|c|c|c|c|c|}
\hline
\multirow{2}*{instance}&\multirow{2}*{nodes}&\multirow{2}*{edges}&\multirow{2}*{mc}&Shor&\multicolumn{3}{c|}{CSSOS}&\multicolumn{3}{c|}{CS-TSSOS}\\
\cline{5-11}
&&&&opt&mb&opt&time&mb&opt&time\\
\hline
g20&$505$&$2045$&$14$&$570$&$120$&$488$&$51.2$&$92$&$488$&$19.6$\\
\hline
g40&$1005$&$3441$&$14$&$1032$&$120$&$885$&$134$&$92$&$893$&$41.1$\\
\hline
g60&$1505$&$4874$&$14$&$1439$&$120$&$1227$&$183$&$92$&$1247$&$71.3$\\
\hline
g80&$2005$&$6035$&$15$&$1899$&$136$&$1638$&$167$&$106$&$1669$&$84.8$\\
\hline
g100&$2505$&$7320$&$14$&$2398$&$120$&$2073$&$262$&$92$&$2128$&$112$\\
\hline
g120&$3005$&$8431$&$14$&$2731$&$120$&$2358$&$221$&$79$&$2443$&$127$\\
\hline
g140&$3505$&$9658$&$13$&$3115$&$105$&$2701$&$250$&$79$&$2812$&$153$\\
\hline
g160&$4005$&$10677$&$14$&$3670$&$120$&$3202$&$294$&$79$&$3404$&$166$\\
\hline
g180&$4505$&$12081$&$13$&$4054$&$105$&$3525$&$354$&$79$&$3666$&$246$\\
\hline
g200&$5005$&$13240$&$13$&$4584$&$105$&$4003$&$374$&$79$&$4218$&$262$\\
\hline
\end{tabular}\\
{\small In this table, only the integer part of optimal values is preserved.}
\end{center}
\end{table}

\subsection{The AC-OPF problem}
The AC optimal power flow (AC-OPF) is a central problem in power systems. It can be formulated as the following POP in complex variables $V_i,S_q^g,S_{ij}$:
\begin{equation}\label{opf}
\begin{cases}
\inf\limits_{V_i,S_q^g,S_{ij}}&\sum_{q\in G}(\mathbf{c}_{2q}(\Re(S_{q}^g))^2+\mathbf{c}_{1q}\Re(S_{q}^g)+\mathbf{c}_{0q})\\
\quad\,\textrm{s.t.}&\angle V_r=0,\\
&\mathbf{S}_{q}^{gl}\le S_{q}^{g}\le \mathbf{S}_{q}^{gu},\quad\forall q\in G,\\
&\boldsymbol{\upsilon}_{i}^l\le|V_i|\le \boldsymbol{\upsilon}_{i}^u,\quad\forall i\in N,\\
&\sum_{q\in G_i}S_q^g-\mathbf{S}_i^d-\mathbf{Y}_i^s|V_{i}|^2=\sum_{(i,j)\in E_i\cup E_i^R}S_{ij},\quad\forall i\in N,\\
&S_{ij}=(\mathbf{Y}_{ij}^*-\mathbf{i}\frac{\mathbf{b}_{ij}^c}{2})\frac{|V_i|^2}{|\mathbf{T}_{ij}|^2}-\mathbf{Y}_{ij}^*\frac{V_iV_j^*}{\mathbf{T}_{ij}},\quad\forall (i,j)\in E,\\
&S_{ji}=(\mathbf{Y}_{ij}^*-\mathbf{i}\frac{\mathbf{b}_{ij}^c}{2})|V_j|^2-\mathbf{Y}_{ij}^*\frac{V_i^*V_j}{\mathbf{T}_{ij}^*},\quad\forall (i,j)\in E,\\
&|S_{ij}|\le\mathbf{s}_{ij}^u,\quad\forall (i,j)\in E\cup E^R,\\
&\boldsymbol{\theta}_{ij}^{\Delta l}\le \angle (V_i V_j^*)\le \boldsymbol{\theta}_{ij}^{\Delta u},\quad\forall (i,j)\in E.\\
\end{cases}
\end{equation}
\revision{The meaning of the symbols in \eqref{opf} is as follows: $N$ - the set of buses, $G$ - the set of generators, $G_i$ - the set of generators connected to bus $i$, $E$ - the set of {\em from} branches, $E^R$ - the set of {\em to} branches, $E_i$ and $E^R_i$ - the subsets of branches that are incident to bus $i$, $\mathbf{i}$ - imaginary unit, $V_i$ - the voltage at bus $i$, $S_q^g$ - the power generation at generator $q$, $S_{ij}$ - the power flow from bus $i$ to bus $j$, $\Re(\cdot)$ - real part of a complex number, $\angle(\cdot)$ - angle of a complex number, $|\cdot|$ - magnitude of a complex number, $(\cdot)^*$ - conjugate of a complex number, $r$ - the voltage angle reference bus. All symbols in boldface are constants ($\mathbf{c}_{0q},\mathbf{c}_{1q},\mathbf{c}_{2q},\boldsymbol{\upsilon}_{i}^l,\boldsymbol{\upsilon}_{i}^u,\mathbf{s}_{ij}^u,\boldsymbol{\theta}_{ij}^{\Delta l},\boldsymbol{\theta}_{ij}^{\Delta u}\in\R$,$\mathbf{S}_{q}^{gl},\mathbf{S}_{q}^{gu},\mathbf{S}_i^d,\mathbf{Y}_i^s,\mathbf{Y}_{ij},\mathbf{b}_{ij}^c,\mathbf{T}_{ij}\in\CC$).} For a full description on the AC-OPF problem, the reader may refer to \cite{baba2019}. By introducing real variables for both real and imaginary parts of each complex variable, we can convert the AC-OPF problem to a POP involving only real variables\footnote{The expressions involving angles of complex variables can be converted to polynomials by using $\tan(\angle z)=y/x$ for $z=x+\mathbf{i}y\in\CC$.}.

To tackle an AC-OPF instance, we first compute a locally optimal solution with a local solver and then rely on an SDP relaxation to certify the global optimality. Suppose that the optimal value reported by the local solver is AC and the optimal value of the SDP relaxation is opt. The {\em optimality gap} between the locally optimal solution and the SDP relaxation is defined by
\begin{equation*}
    \textrm{gap}:=\frac{\textrm{AC}-\textrm{opt}}{\textrm{AC}}\times100\%.
\end{equation*}
If the optimality gap is less than $1.00\%$, then we accept the locally optimal solution as globally optimal. For many AC-OPF instances, the first-order moment-SOS relaxation (Shor's relaxation) is already able to certify the global optimality (with an optimality gap less than $1.00\%$). Therefore, we focus on the more challenging AC-OPF instances for which the optimality gap given by Shor's relaxation is greater than $1.00\%$. 
We select such instances from the AC-OPF library {\em \href{https://github.com/power-grid-lib/pglib-opf}{PGLiB}} \cite{baba2019}.
Since we shall go to the second-order moment-SOS relaxation, we can replace the variables $S_{ij}$ and $S_{ji}$ by their right-hand side expressions in \eqref{opf} and then convert the resulting problem to a POP involving real variables. 
The data for these selected AC-OPF instances are displayed in Table \ref{ac-opf1}, where the AC values are taken from {\em \href{https://github.com/power-grid-lib/pglib-opf}{PGLiB}}.

We solve Shor's relaxation, the CSSOS hierarchy with $d=2$ and the CS-TSSOS hierarchy with $d=2,k=1$ for these AC-OPF instances and the results are displayed in Table \ref{ac-opf1}--\ref{ac-opf2}. For instances 162\_ieee\_dtc, 162\_ieee\_dtc\_api, 500\_tamu, 1888\_rte, with maximal chordal extensions {\tt Mosek} ran out of memory and so we 
use approximately smallest chordal extensions. As the tables show,
CS-TSSOS is more efficient and scales much better with the problem size than CSSOS.
In particular, CS-TSSOS succeeds in reducing the optimality gap to less than $1.00\%$ for all instances.

\begin{table}[htbp]
\caption{The data for AC-OPF instances}\label{ac-opf1}
\begin{center}
\begin{tabular}{|c|c|c|c|c|c|c|}
\hline
\multirow{2}*{case name}&\multirow{2}*{var}&\multirow{2}*{cons}&\multirow{2}*{mc}&\multirow{2}*{AC}&\multicolumn{2}{c|}{Shor}\\
\cline{6-7}
&&&&&opt&gap\\
\hline
3\_lmbd\_api&$12$&$28$&$6$&$1.1242\text{e}4$&$1.0417\text{e}4$&$7.34\%$\\
\hline
5\_pjm&$20$&$55$&$6$&$1.7552\text{e}4$&$1.6634\text{e}4$&$5.22\%$\\
\hline
24\_ieee\_rts\_api&$114$&$315$&$10$&$1.3495\text{e}5$&$1.3216\text{e}5$&$2.06\%$\\
\hline
24\_ieee\_rts\_sad&$114$&$315$&$14$&$7.6943\text{e}4$&$7.3592\text{e}4$&$4.36\%$\\
\hline
30\_as\_api&$72$&$297$&$8$&$4.9962\text{e}3$&$4.9256\text{e}3$&$1.41\%$\\
\hline
73\_ieee\_rts\_api&$344$&$971$&$16$&$4.2263\text{e}5$&$4.1041\text{e}5$&$2.89\%$\\
\hline
73\_ieee\_rts\_sad&$344$&$971$&$16$&$2.2775\text{e}5$&$2.2148\text{e}5$&$2.75\%$\\
\hline
162\_ieee\_dtc&$348$&$1809$&$21$&$1.0808\text{e}5$&$1.0616\text{e}5$&$1.78\%$\\
\hline
162\_ieee\_dtc\_api&$348$&$1809$&$21$&$1.2100\text{e}5$&$1.1928\text{e}5$&$1.42\%$\\
\hline
240\_pserc&$766$&$3322$&$16$&$3.3297\text{e}6$&$3.2818\text{e}6$&$1.44\%$\\
\hline
500\_tamu\_api&$1112$&$4613$&$20$&$4.2776\text{e}4$&$4.2286\text{e}4$&$1.14\%$\\
\hline
500\_tamu&$1112$&$4613$&$30$&$7.2578\text{e}4$&$7.1034\text{e}4$&$2.12\%$\\
\hline
793\_goc&$1780$&$7019$&$18$&$2.6020\text{e}5$&$2.5636\text{e}5$&$1.47\%$\\
\hline
1888\_rte&$4356$&$18257$&$26$&$1.4025\text{e}6$&$1.3748\text{e}6$&$1.97\%$\\
\hline
3022\_goc&$6698$&$29283$&$50$&$6.0143\text{e}5$&$5.9278\text{e}5$&$1.44\%$\\
\hline
\end{tabular}
\end{center}
\end{table}

\begin{table}[htbp]
\caption{The result for AC-OPF instances}\label{ac-opf2}
\begin{center}{\small
\begin{tabular}{|c|c|c|c|c|c|c|c|c|c|}
\hline
\multirow{2}*{case name}&\multicolumn{4}{c|}{CSSOS}&\multicolumn{5}{c|}{CS-TSSOS}\\
\cline{2-10}
&mb&opt&time&gap&mb&opt&time&gap&CE\\
\hline
3\_lmbd\_api&$28$&$1.1242\text{e}4$&$0.21$&$0.00\%$&$22$&$1.1242\text{e}4$&$0.09$&$0.00\%$&max\\
\hline
5\_pjm&$28$&$1.7543\text{e}4$&$0.56$&$0.05\%$&$22$&$1.7543\text{e}4$&$0.30$&$0.05\%$&max\\
\hline
24\_ieee\_rts\_api&$66$&$1.3442\text{e}5$&$5.59$&$0.39\%$&$31$&$1.3396\text{e}5$&$2.01$&$0.73\%$&max\\
\hline
24\_ieee\_rts\_sad&$120$&$7.6943\text{e}4$&$94.9$&$0.00\%$&$39$&$7.6942\text{e}4$&$14.8$&$0.00\%$&max\\
\hline
30\_as\_api&$45$&$4.9927\text{e}3$&$4.43$&$0.07\%$&$22$&$4.9920\text{e}3$&$2.69$&$0.08\%$&max\\
\hline
73\_ieee\_rts\_api&$153$&$4.2246\text{e}5$&$758$&$0.04\%$&$44$&$4.2072\text{e}5$&$96.0$&$0.45\%$&max\\
\hline
73\_ieee\_rts\_sad&$153$&$2.2775\text{e}5$&$504$&$0.00\%$&$44$&$2.2766\text{e}5$&$71.5$&$0.04\%$&max\\
\hline
162\_ieee\_dtc&$253$&$-$&$-$&$-$&$34$&$1.0802\text{e}5$&$278$&$0.05\%$&min\\
\hline
162\_ieee\_dtc\_api&$253$&$-$&$-$&$-$&$34$&$1.2096\text{e}5$&$201$&$0.03\%$&min\\
\hline
240\_pserc&$153$&$3.3072\text{e}6$&$585$&$0.68\%$&$44$&$3.3042\text{e}6$&$33.9$&$0.77\%$&max\\
\hline
500\_tamu\_api&$231$&$4.2413\text{e}4$&$3114$&$0.85\%$&$39$&$4.2408\text{e}4$&$46.6$&$0.86\%$&max\\
\hline
500\_tamu&$496$&$-$&$-$&$-$&$31$&$7.2396\text{e}4$&$410$&$0.25\%$&min\\
\hline
793\_goc&$190$&$2.5938\text{e}5$&$563$&$0.31\%$&$33$&$2.5932\text{e}5$&$66.1$&$0.34\%$&max\\
\hline
1888\_rte&$378$&$-$&$-$&$-$&$27$&$1.3953\text{e}6$&$934$&$0.51\%$&min\\
\hline
3022\_goc&$1326$&$-$&$-$&$-$&$76$&$5.9858\text{e}5$&$1886$&$0.47\%$&max\\
\hline
\end{tabular}}
\end{center}
\end{table}

\section{Discussion and conclusions}\label{conc}

This paper introduces the CS-TSSOS hierarchy, a sparse variant of the moment-SOS hierarchy, which can be used to solve large-scale real-world nonlinear optimization problems, assuming that the input data are sparse polynomials. 
In addition to its theoretical convergence guarantees, CS-TSSOS allows one to make a trade-off between the quality of optimal values and the computational efficiency by controlling the types of chordal extensions and the sparse order $k$. 

By fully exploiting sparsity, CS-TSSOS allows one to go beyond Shor's relaxation and solve the second-order moment-SOS relaxation associated with large-scale POPs to obtain more accurate bounds. Indeed CS-TSSOS can handle second-order relaxations of POP instances with thousands of variables and constraints on a standard laptop in tens of minutes. Such instances include the optimal power flow (OPF) problem, an important challenge in the management of electricity networks. In particular, our plan
is to perform advanced numerical experiments on HPC cluster, for OPF instances with larger numbers of buses \cite{eltved2019robustness}.

This work suggests additional investigation tracks for further research:

1) The standard procedure of extracting optimal solutions for the dense moment-SOS hierarchy does not apply to the CS-TSSOS hierarchy. It would be interesting to develop a procedure for extracting (approximate) solutions from partial information of moment matrices.

2) Recall that chordal extension plays an important role for both correlative and term sparsity patterns. It turns out that the size of the resulting maximal cliques 
is crucial for the overall computational efficiency of the CS-TSSOS hierarchy. So far, we have only considered \emph{maximal} chordal extensions (for convergence guarantee) and approximately \emph{smallest} chordal extensions. It would be worth investigating more general choices of chordal extensions.

3) The CS-TSSOS strategy could be adapted to other applications involving sparse polynomial problems, including deep learning  \cite{chen2020polynomial} or noncommutative optimization problems  \cite{klep2019sparse} arising in quantum information.

4) At last but not least, a challenging research issue is 
to establish serious computationally cheaper alternatives to interior-point methods for solving SDP relaxations of POPs. The recent work \cite{yurtsever2021scalable} which reports spectacular results for standard SDPs (and Max-Cut problems in particular) is a positive sign in this direction.

~

\paragraph{\textbf{Acknowledgements}.} 
We would like to thank Tillmann Weisser for helpful discussions on OPF problems.
The first and second authors were supported by the Tremplin ERC Stg Grant ANR-18-ERC2-0004-01 (T-COPS project).
The second author was supported by the FMJH Program PGMO (EPICS project) and  EDF, Thales, Orange et Criteo.
This work has benefited from  the European Union's Horizon 2020 research and innovation programme under the Marie Sklodowska-Curie Actions, grant agreement 813211 (POEMA) as well as from the AI Interdisciplinary Institute ANITI funding, through the French ``Investing for the Future PIA3'' program under the Grant agreement n$^{\circ}$ANR-19-PI3A-0004.
The third author was also supported by the European Research Council (ERC) under the European's Union Horizon 2020 research and innovation program (grant agreement 666981 TAMING).

\bibliographystyle{alpha}
\bibliography{references}
\end{document}